\documentclass[letterpaper,11pt,twoside,keywordsasfootnote,addressatend,noinfoline]{article}
\usepackage{fullpage}
\usepackage[english]{babel}
\usepackage{amssymb}
\usepackage{amsmath}
\usepackage{theorem}
\usepackage{epsfig}
\usepackage{imsart}

\theorembodyfont{\slshape}
\newtheorem{theor}{Theorem}

\newtheorem{lemma}[theor]{Lemma}

\newcommand{\Z}{\mathbb{Z}}
\newcommand{\R}{\mathbb{R}}
\newcommand{\w}{\mathbf{w}}
\newcommand{\x}{\mathbf{x}}
\newcommand{\y}{\mathbf{y}}
\newcommand{\z}{\mathbf{z}}
\newcommand{\0}{\mathbf{0}}
\newcommand{\ep}{\epsilon}
\newcommand{\ind}{\mathbf 1}

\DeclareMathOperator{\uniform}{Uniform \,}
\DeclareMathOperator{\dist}{dist}

\newenvironment{proof}{\noindent{\scshape Proof.}}{\hspace{2mm} $\square$}

\begin{document}

\begin{frontmatter}

\title     {Survival and extinction results for a patch \\ model with sexual reproduction}
\runtitle  {Survival and extinction results for a patch model with sexual reproduction}
\author    {N. Lanchier\thanks{Research supported in part by NSF Grant DMS-10-05282.}}
\runauthor {N. Lanchier}
\address   {School of Mathematical and Statistical Sciences, \\ Arizona State University, \\ Tempe, AZ 85287, USA.}

\begin{abstract} \ \
 This article is concerned with a version of the contact process with sexual reproduction on a graph with two levels
 of interactions modeling metapopulations.
 The population is spatially distributed into patches and offspring are produced in each patch at a rate proportional
 to the number of pairs of individuals in the patch (sexual reproduction) rather than simply the number of individuals
 as in the basic contact process.
 Offspring produced at a given patch either stay in their parents' patch or are sent to a nearby patch with some fixed
 probabilities.
 Specifically, we prove lower and upper bounds for the probability of long-term survival for the process starting with a
 single fully occupied patch.
 Our main finding is that, with probability close to one and for a certain set of parameters, the metapopulation survives
 in the presence of nearest neighbor interactions while it dies out in the presence of long range interactions, suggesting
 that the best strategy for the population to expand in space is to use intermediate dispersal ranges.
 This result is due to the presence of a so-called strong Allee effect induced by the birth mechanism and does not hold
 for the basic contact process.
\end{abstract}



\end{frontmatter}


\section{Introduction}
\label{sec:intro}

\indent The term Allee effect refers to a certain process that leads to decreasing net population growth with decreasing
 density \cite{allee_1931}.
 In vase the growth rate becomes negative at low density, this monotone relationship results in the existence of a so-called Allee
 threshold below which populations are at high risk of being driven toward extinction.
 This phenomenon may be due to various ecological factors:
 failure to locate mates, inbreeding depression, failure to satiate predators, lack of cooperative feeding, etc.
 Research on this topic is copious and is reviewed in~\cite{courchamp_berec_gascoigne_2009} but rigorous mathematical analyses of
 stochastic spatial models that include an Allee effect are much more limited.

\indent In the model proposed by Borrello \cite{borrello_2012}, each site of the infinite regular lattice represents a patch that
 can host a local population, and a strong Allee effect is included in the form of a varying individual death rate taking a larger
 value for local populations below some threshold.
 The model is used to show that, when only small flocks of individuals can migrate from patch to patch, the metapopulation goes
 extinct whereas survival is possible if large enough flocks of individuals can migrate.
 The author has also used the framework of interacting particle systems to study the consequence of an Allee effect \cite{lanchier_2013}.
 His model is a modification of the averaging process that also includes a threshold:
 local populations below this threshold go extinct whereas local populations above this threshold expand to their carrying capacity,
 each at rate one.
 The key component in this work is the topology of the network of interactions rather than the size of the migrating flocks, and
 the analysis of the process starting from a single occupied patch on various graphs indicates that the probability of long-term
 survival of the metapopulation decreases to zero as the degree of the network of interactions increases to infinity.
 This result suggests that long range dispersal promotes extinction of metapopulations subject to a strong Allee effect.

\indent The modeling approach of the present paper is somewhat different.
 The model we propose is a version of the contact process with sexual reproduction \cite{neuhauser_1994, noble_1992} on a
 graph that includes two levels of interactions modeling metapopulations:
 individuals are produced within each patch at a rate proportional to the number of pairs of individuals in the patch rather
 than a rate simply proportional to the number of individuals in the patch.
 This birth mechanism (sexual reproduction) reflects the difficulty to locate mates in patches at low density, which has been identified
 by ecologists as one of the most common causes of Allee effect.
 In particular, while the Allee effect is forced into the model in the form of a threshold parameter in
 both \cite{borrello_2012, lanchier_2013}, it is on the contrary naturally induced by the birth mechanism in the model considered
 in this paper.
 Using block constructions and duality techniques together with large deviation estimates and martingale theory, the main objective
 is to study the survival probability of the process starting with a single fully occupied patch.


\section{Model description and main results}
\label{sec:results}

\indent For simplicity and to avoid complicated expressions later, we focus on the one-dimensional case, but all our results easily
 extend to higher dimensions.
 Each integer $x \in \Z$ represents a patch that can host a local population of up to $N$ individuals, and it is convenient to think
 of each site as a set of spatial locations that can be either empty or occupied by one individual.
 By convention, we use bold letters to denote these spatial locations:
 $$ \x := (x, j) \in \Z \times K_N := \Z \times \{1, 2, \ldots, N \} $$
 are the possible spatial locations at site/patch $x$.
 The model we consider is a continuous-time Markov chain whose state at time $t$ is a spatial configuration
 $$ \eta_t : \Z \times K_N \longrightarrow \{0, 1 \} \quad \hbox{where} \quad 0 = \hbox{empty} \quad \hbox{and} \quad 1 = \hbox{occupied}. $$
 Each individual dies at rate one.
 Offspring produced at a given patch are sent to a spatial location chosen uniformly at random either from the parents' patch or from
 a neighboring patch.
 In either case, offspring are produced at a rate proportional to $N$ times the fraction of occupied pairs of spatial locations and
 the birth is suppressed if the target location is already occupied.
 The proportionality constant is denoted by $a$ for offspring sent within their parents' patch and by $b$ for offspring sent outside
 their parents' patch.
 Sexual reproduction is modeled by the fact that the birth rate is related to the number of occupied pairs and the reason for multiplying
 by $N$ is to have births and deaths occurring at the same time scale.
 To define the dynamics more formally, we write
 $$ x \sim y \quad \hbox{if and only if} \quad x \neq y \ \hbox{and} \ |x - y| \leq M $$
 where $M$ is the dispersal range, and define the projection map
 $$ \pi \ : \ \x := (x, j) \in \Z \times K_N \ \mapsto \ \pi (\x) := x \in \Z. $$
 For all $\x \in \Z \times K_N$, we also define the sets
 $$ \begin{array}{rcl}
     A (\x) & := & \hbox{set of potential parents' pairs within the patch containing $\x$} \vspace*{3pt} \\
            & := & \{(\y, \z) \in (\Z \times K_N)^2 : \y \neq \z \ \hbox{and} \ \pi (\x) = \pi (\y) = \pi (\z) \} \vspace*{9pt} \\
     B (\x) & := & \hbox{set of potential parents' pairs in the neighborhood of the patch containing $\x$} \vspace*{3pt} \\
            & := & \{(\y, \z) \in (\Z \times K_N)^2 : \y \neq \z \ \hbox{and} \ \pi (\x) \sim \pi (\y) = \pi (\z) \}. \end{array} $$
 The dynamics is then described by the Markov generator
\begin{equation}
\label{eq:micro-model}
  \begin{array}{rcl}
   L_- f (\eta) & = & \displaystyle \sum_{\x} \ \ [f (\eta_{\x, 0}) - f (\eta)] \vspace*{4pt} \\ & + &
                      \displaystyle \sum_{\x} \ \bigg(\frac{a}{N (N - 1)} \sum_{(\y, \z) \in A (\x)} \eta (\y) \,\eta (\z) \bigg) \
                       [f (\eta_{\x, 1}) - f (\eta)] \vspace*{4pt} \\ & + &
                      \displaystyle \sum_{\x} \ \bigg(\frac{1}{2M} \ \frac{b}{N (N - 1)} \sum_{(\y, \z) \in B (\x)} \eta (\y) \,\eta (\z) \bigg) \
                       [f (\eta_{\x, 1}) - f (\eta)] \end{array}
\end{equation}
 where configuration $\eta_{\x, i}$ is obtained from $\eta$ by setting the state at $\x$ to $i$.
 We shall call this process the microscopic representation.
 To study the model, it is convenient to also consider its mesoscopic representation that keeps track of the metapopulation at the patch
 level rather than at the individual level.
 This new process is simply obtained by setting
 $$ \xi_t (x) \ := \sum_{\pi (\x) = x} \eta_t (\x) \quad \hbox{for all} \quad (x, t) \in \Z^d \times \R_+. $$
 In words, the process counts the number of individuals at each patch.
 Note that, since the particular locations of the individuals within each patch is unimportant from a dynamical point of view, this new
 process is again a Markov process, and its Markov generator is given by
\begin{equation}
\label{eq:macro-model}
  \begin{array}{rcl}
   L_+ f (\xi) & = & \displaystyle \sum_x \ \xi (x) \,[f (\xi_{x-}) - f (\xi)] \vspace*{4pt} \\ & + &
                     \displaystyle \sum_x \ \frac{a}{N (N - 1)} \ \ \xi (x) \,(\xi (x) - 1)(N - \xi (x)) \ [f (\xi_{x+}) - f (\xi)] \vspace*{4pt} \\ & + &
                     \displaystyle \sum_x \ \sum_{y \sim x} \ \frac{1}{2M} \ \frac{b}{N (N - 1)} \ \ \xi (y) \,(\xi (y) - 1)(N - \xi (x)) \ [f (\xi_{x+}) - f (\xi)] \end{array}
\end{equation}
 where configuration $\xi_{x \pm}$ is obtained from $\xi$ by adding/removing one individual at $x$.
 The analog of this model derived from the basic contact process rather than the contact process with sexual reproduction has been
 studied in \cite{bertacchi_lanchier_zucca_2011} where it is proved that
\begin{itemize}
 \item the process survives when $a + b > 1$ and $N$ is sufficiently large \vspace*{2pt}
 \item the process dies out for all values of the parameter $N$ when $a + b \leq 1$.
\end{itemize}
 The analysis of the model \eqref{eq:micro-model}--\eqref{eq:macro-model} is more challenging.
 Constructing the system graphically from a collection of independent Poisson processes and using standard coupling arguments, we easily
 prove that the process is nondecreasing with respect to the two birth parameters $a$ and $b$, and is attractive.
 In particular, the survival probability
 $$ P \,(\eta_t \neq \0 \ \hbox{for all} \ t > 0) \ = \ P \,(\xi_t \neq \0 \ \hbox{for all} \ t > 0) $$
 where $\0$ stands for the all empty configuration is nondecreasing with respect to $a$ and $b$ but also with respect to the initial
 configuration.
 Note however that basic coupling arguments do not imply monotonicity with respect to $N$ or the dispersal range $M$.
 Attractiveness implies that the limiting distribution of the process starting from the all occupied configuration exists.
 It is usually called the upper invariant measure.
 Survival and extinction can be studied through this invariant measure looking at whether it is a nontrivial distribution or the point
 mass at configuration $\0$.
 However, to answer an important ecological question, namely, whether an alien species already established in one patch either
 successfully expands in space or is doomed to extinction, instead of looking at the upper invariant measure, we study the probability
 of long-term survival for the process starting with a single fully occupied patch defined as
\begin{equation}
\label{eq:survival-probability}
  P \,(\xi_t \neq \0 \ \hbox{for all} \ t > 0 \ | \ \xi_0 = \ind_x)
\end{equation}
 where $\ind_x$ is the configuration
 $$ \ind_x : \Z \longrightarrow \{0, 1, \ldots, N \} \quad \hbox{with} \quad \ind_x (x) = N \quad \hbox{and} \quad \ind_x (y) = 0 \ \hbox{for} \ y \neq x. $$
 To state but also motivate our results, we first observe that the time evolution of the fraction of individuals in a given patch is
 related to the differential equation
\begin{equation}
\label{eq:mean-field}
  u' (t) \ = \ Q (u) \quad \hbox{where} \quad Q (X) = a X^2 (1 - X) - X
\end{equation}
 with initial condition $u (0) \in (0, 1)$.
 This model exhibits two possible regimes:
\begin{enumerate}
 \item when $a < 4$, there is extinction: $u (t) \to 0$ as time $t \to \infty$. \vspace*{4pt}
 \item when $a > 4$, the polynomial $Q$ has two nontrivial roots given by
  $$ c_- \ := \ c_- (a) \ = \ \frac{1}{2} - \sqrt{\frac{1}{4} - \frac{1}{a}} \qquad \hbox{and} \qquad
     c_+ \ := \ c_+ (a) \ = \ \frac{1}{2} + \sqrt{\frac{1}{4} - \frac{1}{a}} $$
  and the system is bistable: $c_- (a)$ is unstable and
  $$ u (t) \to 0 \ \ \hbox{when} \ \ u (0) < c_- (a) \qquad \hbox{whereas} \qquad u (t) \to c_+ \ \ \hbox{when} \ \ u (0) > c_- (a). $$
\end{enumerate}
 Model \eqref{eq:mean-field} is probably the simplest model of population subject to a strong Allee effect.
 It shows that the value $a = 4$ is a critical value.
 At least when $N$ is large, which is the case of interest in ecology, our results for the stochastic spatial model reveal again
 different behaviors depending on whether the inner birth parameter $a$ and more generally the overall birth parameter $a + b$ is
 either larger or smaller than four.
 Our first theorem gives explicit conditions on the two birth parameters for the metapopulation to survive with probability close to
 one when $N$ is large and offspring can only be sent to the two patches adjacent to the parents' patch.
\begin{theor} --
\label{th:survival}
 Assume that $M = 1$ and that either
 $$ \hbox{(A1)} \ \ b > 8 \qquad \hbox{or} \qquad \hbox{(A2)} \ \ a > 4 \ \hbox{and} \ b > 2 a^3 \,c_-^4. $$
 Then, there exists $c > 0$ such that
 $$ P \,(\xi_t = \0 \ \hbox{for some} \ t > 0 \ | \ \xi_0 = \ind_x) \ \leq \ \exp (- c \sqrt N) \quad \hbox{for all $N$ sufficiently large}. $$
\end{theor}
 Based on our simple analysis of \eqref{eq:mean-field}, we conjecture that $b > 4$ is a sufficient condition for the existence of a
 nontrivial invariant measure when $N$ is large, but we believe that this condition is not sufficient for the survival probability
 starting with a single occupied patch to be close to one.
 The reason is that only half of the offspring produced at the source patch are sent to a given adjacent patch so the value 8 can
 be understood literally as four times two adjacent patches.
 To comment on the second assumption, we note that the condition $a > 4$ alone is not sufficient even for the existence of a nontrivial
 invariant measure since, in the absence of migrations, each local population is doomed to extinction.
 However, using the expression of the root $c_-$ to compute explicitly the bound in condition (A2), we find for instance that the theorem
 holds for the pairs
 $$ (a > 4 \ \hbox{and} \ b > 8) \quad \hbox{and} \quad
    (a > 5 \ \hbox{and} \ b > 1.459) \quad \hbox{and} \quad
    (a > 6 \ \hbox{and} \ b > 0.862) $$
 indicating that the bound for $b$ decreases fast with respect to $a$.
 We now focus on extinction results.
 Conditions for extinction of the basic contact process can be easily found by comparing the process with a branching process.
 The same approach can be used to find conditions for extinction of the process with sexual reproduction and two levels of interactions.
 Indeed, observing that, in patches with exactly~$j$ individuals, offspring are produced at rate
 $$ (a + b) \,N \ \frac{j \,(j - 1)}{N \,(N - 1)} \ \leq \ (a + b) \,N \ \bigg(\frac{j}{N} \bigg)^2 \ \leq \ (a + b) \,j $$
 we deduce that the number of individuals in the metapopulation is dominated stochastically by a simple birth and death process with
 birth parameter $ a + b$.
 In particular, there is almost sure extinction whenever $a + b \leq 1$.
 This condition, however, is far from being optimal for the contact process with sexual reproduction.
 Indeed, the next theorem and the analysis of \eqref{eq:mean-field} suggest that extinction occurs whenever $a + b < 4$, a condition
 that we believe to be optimal.
\begin{theor} --
\label{th:extinction}
 Let $a + b < 4$.
 Then, for all $\ep > 0$, there exists $t$ such that
 $$ P \,(\eta_t (\x) = 1) \ \leq \ \ep \quad \hbox{for all $N$ sufficiently large}. $$
\end{theor}
 The proof is based on duality techniques, and more precisely the fact that the dual process of the microscopic model \eqref{eq:micro-model}
 dies out almost surely when $a + b < 4$.
 In particular, the theorem holds even for the process starting from the all occupied configuration.
 Note however that this result only suggests but does not imply extinction because it only holds for parameters $N$ that depend on time~$t$.
 This is due to the fact that the survival probability of the somewhat complicated dual process can only be computed explicitly in the absence
 of collisions in the dual process which, in turn, can only be obtained in the limit as $N \to \infty$.
 Our last theorem also looks at the extinction regime and is probably the most interesting one from an ecological perspective.
 It shows that, in contrast with our first result that identifies parameter regions in which survival occurs with probability arbitrarily
 close to one, regardless of the choice of the birth parameters, extinction occurs with probability close to one when the dispersal range
 $M$ is sufficiently large.
\begin{theor} --
\label{th:range}
 For all $M$ large
\begin{enumerate}
 \item and all $a < 4$, we have \vspace*{4pt}
 \item[] \hspace*{20pt} $P \,(\xi_t \neq \0 \ \hbox{for all} \ t > 0 \ | \ \xi_0 = \ind_x) \ \leq \ M^{-1/3} \,(1/2 + b \,N (1 - a/4)^{-1})$. \vspace*{4pt}
 \item and $a = 4$, we have \vspace*{4pt}
 \item[] \hspace*{20pt} $P \,(\xi_t \neq \0 \ \hbox{for all} \ t > 0 \ | \ \xi_0 = \ind_x) \ \leq \ M^{-1/3} \,(1/2 + (b/2)(N + 2)^2)$. \vspace*{4pt}
 \item and all $a > 4$, we have \vspace*{4pt}
 \item[] \hspace*{20pt} $P \,(\xi_t \neq \0 \ \hbox{for all} \ t > 0 \ | \ \xi_0 = \ind_x) \ \leq \ M^{-1/3} \,(1/2 + b \,(a/4 - 1)^{-2} \,(a/4)^{N + 2})$.
\end{enumerate}
\end{theor}
 There are three different estimates because the survival probability is related to the time to extinction of a patch in isolation,
 which scales differently depending on whether the inner birth rate is subcritical, critical or supercritical.
 In either case, the theorem shows that the survival probability decreases to zero as the dispersal range increases indicating that,
 though dispersal is necessary for expansion in space, long range dispersal promotes extinction of metapopulations subject to a strong
 Allee effect caused by sexual reproduction.
 In particular, the effects of dispersal are opposite for the process with and without sexual reproduction since, in the presence
 of long range dispersal, the basic contact process approaches a branching process with critical values for
 survival significantly smaller than that of the process with nearest neighbor interactions.


\section{Proof of Theorem \ref{th:survival} under assumption A1}
\label{sec:outer-survival}

\indent This section and the next section are devoted to the proof of Theorem \ref{th:survival}.
 Key ingredients to establish the theorem are the identification of submartingales and supermartingales, the application of the optimal
 stopping theorem to these processes, as well as large deviation estimates for various random variables.
 This section proves survival under the assumption: $b > 8$.
 To begin with, we give an overview of the main steps of the proof. Let
 $$ X_t \ := \ \xi_t (x) \quad \hbox{and} \quad Y_t \ := \ \xi_t (y) $$
 where $x, y \in \Z$ are adjacent patches: $|x - y| = 1$.
 Also, for all $\rho > 0$, we define
 $$ \Omega_{\rho} \ := \ \{(i, j) \in \{0, 1, \ldots, N \}^2 : i > N/2 + \rho \sqrt N \ \hbox{and} \ j > N/2 + \rho \sqrt N \}. $$
 These are considered in the proof as subsets of the state space of $(X_t, Y_t)$.
 The proof is divided into four steps which we now briefly describe. \vspace*{4pt} \\
\noindent {\bf Step 1} -- Starting with patch $x$ fully occupied and patch $y$ empty, the overall population in both patches increases
 while both population sizes to converge to each other resulting in the number of individuals in each patch being quickly above a certain threshold:
 $(X, Y) \in \Omega_2$ quickly. \vspace*{4pt} \\
\noindent {\bf Step 2} -- Once both local populations are above this threshold, they stay at least around this threshold for a long time:
 $(X, Y) \in \Omega_0$ for a long time after hitting $\Omega_2 \subset \Omega_0$. \vspace*{4pt} \\
\noindent {\bf Step 3} -- From the previous two steps, we deduce that, with probability arbitrarily close to one when the parameter $N$ is large,
 both patches are above some threshold at some deterministic time, an event that we shall call successful first invasion. \vspace*{4pt} \\
\noindent {\bf Step 4} -- The argument to prove that subsequent invasions also occur with probability close to one is slightly different
 because the probability that both patches $x$ and $y$ are fully occupied at some deterministic time is close to zero.
 To deal with subsequent invasions, we adapt step 3 to deduce that, if the populations at patches $x$ and $y$ lie above some
 threshold then the same holds in the two nearby patches after some deterministic time. \vspace*{4pt} \\
 The first part of the theorem is then obtained by combining the last two steps with a block construction in order to compare the process
 properly rescaled in space and time with oriented percolation.
 To prove the first step, we define the sets
 $$ \begin{array}{rcl}
    \Omega_- & := & \{(i, j) \in \{0, 1, \ldots, N \}^2 : i + j > N - 4 \sqrt N \ \hbox{and} \ j < N/2 + 2 \sqrt N \ \hbox{and} \ i > j \} \vspace*{2pt} \\
    \Omega_+ & := & \{(i, j) \in \{0, 1, \ldots, N \}^2 : i + j > N - 4 \sqrt N \ \hbox{and} \ i < N/2 + 2 \sqrt N \ \hbox{and} \ i < j \} \end{array} $$
 which are represented in Figure \ref{fig:sketch}.
 Step 1 is proved in Lemmas~\ref{lem:outer-sum}--\ref{lem:outer-fast}.

\begin{figure}[t]
\centering
\scalebox{0.50}{\input{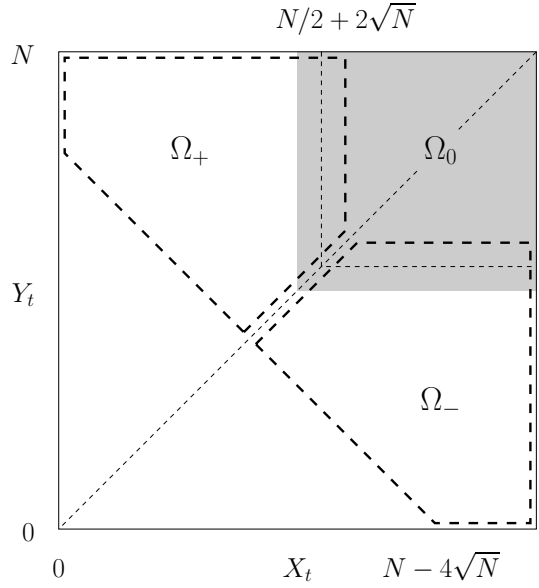}}
\caption{\upshape{Picture related to the proof of Theorem \ref{th:survival}.}}
\label{fig:sketch}
\end{figure}


\begin{lemma} --
\label{lem:outer-sum}
 Let $b > 8$.
 For all $N$ large,
 $$ \begin{array}{l} \lim_{\,h \to 0} \,h^{-1} \,E \,((X_{t + h} + Y_{t + h}) - (X_t + Y_t) \ | \ (X_t, Y_t) \in \Omega_- \cup \,\Omega_+) \ \geq \ N (1/4)(b/8 - 1). \end{array} $$
\end{lemma}
\begin{proof}
 Each $(i, j) \in \Omega_-$ can be written as
 $$ i \ := \ N/2 - 2 \sqrt N + u \quad \hbox{and} \quad j \ := \ N/2 + 2 \sqrt N - v $$
 where $u, v \in (0, N/2 + 2 \sqrt N)$.
 It follows that
 $$ \begin{array}{l}
     \lim_{\,h \to 0} \,h^{-1} \,E \,((X_{t + h} + Y_{t + h}) - (X_t + Y_t) \ | \ (X_t, Y_t) = (i, j) \in \Omega_-) \vspace*{4pt} \\ \hspace*{25pt}
        = \ N^{-2} \,(b/2)(i^2 \,(N - j) + j^2 \,(N - i)) - (i + j) + O (1) \vspace*{4pt} \\ \hspace*{25pt}
        = \ (b/8)(2 N u^2 + 2 N v^2 + N^2 u - N^2 v + 8 N uv + 4 u^2v - 4 uv^2 + N^3) \vspace*{4pt} \\ \hspace*{50pt} - \ (N + u - v) + O (\sqrt N) \vspace*{4pt} \\ \hspace*{25pt}
        = \ (b/8)(2 N (u^2 + v^2) + N^2 (u - v) + 4 uv \,(2N + u - v) + N^3) \vspace*{4pt} \\ \hspace*{50pt} - \ (N + u - v) + O (\sqrt N) \vspace*{4pt} \\ \hspace*{25pt}
        = \ (b/8)(2 N (u^2 + v^2) + 4 uv \,(2 N + u - v)) \vspace*{4pt} \\ \hspace*{50pt} + \ (b/8 - 1)(N + u - v) + O (\sqrt N) \ \geq \ N (1/4)(b/8 - 1) \end{array} $$
 for all $N$ large.
 By symmetry, we also have
 $$ \begin{array}{l}
     \lim_{\,h \to 0} \,h^{-1} \,E \,((X_{t + h} + Y_{t + h}) - (X_t + Y_t) \ | \ (X_t, Y_t) \in \Omega_+) \ \geq \ N (1/4)(b/8 - 1) \end{array} $$
 which completes the proof.
\end{proof}
\begin{lemma} --
\label{lem:outer-supermartingale}
 Let $b > 8$.
 There exists $c_1 > 0$ such that
 $$ \begin{array}{l} \lim_{\,h \to 0} \,h^{-1} \,E \,(Z_{t + h} - Z_t \ | \ (X_t, Y_t) \in \Omega_- \cup \,\Omega_+) \ \leq \ 0 \quad \hbox{for all $N$ large} \end{array} $$
 where $Z_t := \exp \,(- c_1 (X_t + Y_t))$.
\end{lemma}
\begin{proof}
 To begin with, we define the functions
 $$ \begin{array}{l}
     \phi_{ij} (c) \ := \ \lim_{\,h \to 0} \,h^{-1} \,E \,(\exp \,(- c \,(X_{t + h} + Y_{t + h})) \vspace*{4pt} \\ \hspace*{80pt} - \exp \,(- c \,(X_t + Y_t)) \ | \ (X_t, Y_t) = (i, j)) \end{array} $$
 and notice that
 $$ \begin{array}{l}
     \phi_{ij}' (c) \ := \ \lim_{\,h \to 0} \,h^{-1} \,E \,(- (X_{t + h} + Y_{t + h}) \,\exp \,(- c \,(X_{t + h} + Y_{t + h})) \vspace*{4pt} \\ \hspace*{100pt} + \
                            (X_t + Y_t) \,\exp \,(- c \,(X_t + Y_t)) \ | \ (X_t, Y_t) = (i, j)). \end{array} $$
 This, together with Lemma~\ref{lem:outer-sum}, implies that
 $$ \phi_{ij}' (0) \ \leq \ - c_2 \ := \ - N (1/4)(b/8 - 1) \quad \hbox{for all} \quad (i, j) \in \Omega_- \cup \,\Omega_+. $$
 Using also that $\phi_{ij} (0) = 0$, we get the existence of $c_1 > 0$ such that
 $$ \begin{array}{l}
     \phi_{ij} (c_1) \ = \ \lim_{\,h \to 0} \,h^{-1} \,E \,(\exp \,(- c_1 (X_{t + h} + Y_{t + h})) \vspace*{4pt} \\ \hspace*{80pt} - \
                           \exp \,(- c_1 (X_t + Y_t)) \ | \ (X_t, Y_t) = (i, j)) \ \leq \ 0 \end{array} $$
 for all $(i, j) \in \Omega_- \cup \,\Omega_+$.
 This completes the proof.
\end{proof} \\ \\
 Relying on Lemmas~\ref{lem:outer-sum}--\ref{lem:outer-supermartingale}, we can now complete the proof of the first step.
 To state our result, it is convenient to introduce the hitting time
 $$ \tau_{\rho}^+ \ := \ \inf \,\{t : (X_t, Y_t) \in \Omega_{\rho} \}. $$
 In addition to Lemmas~\ref{lem:outer-sum}--\ref{lem:outer-supermartingale}, the proof also relies on an application of the
 optimal stopping theorem for supermartingales and large deviation estimates for the Poisson distribution.
\begin{lemma} --
\label{lem:outer-fast}
 Let $b > 8$.
 There exists $c_3 > 0$ such that
 $$ P \,(\tau_2^+ > N \ | \ X_0 + Y_0 \geq N) \ \leq \ \exp (- c_3 \sqrt N) \quad \hbox{for all $N$ large}. $$
\end{lemma}
\begin{proof}
 First, we introduce the stopping times
 $$ \tau \ := \ \inf \,\{t : X_t + Y_t < N - 4 \sqrt N \} \ = \ \inf \,\{t : Z_t > \exp (- c_1 \,(N - 4 \sqrt N)) \} $$
 as well as $\sigma := \min (\tau, \tau_2^+)$. Then,
\begin{equation}
\label{eq:outer-fast-1}
  \begin{array}{rcl}
    E \,(Z_{\sigma}) & = &
    E \,(Z_{\tau}) \,P \,(\tau < \tau_2^+) + E \,(Z_{\tau_2^+}) \,P \,(\tau > \tau_2^+) \vspace*{4pt} \\ & \geq &
      \exp (- c_1 (N - 4 \sqrt N)) \,P \,(\tau < \tau_2^+) + \exp (- 2 \,c_1 N) \,P \,(\tau > \tau_2^+) \vspace*{4pt} \\ & = &
      \exp (- 2 \,c_1 N) + (\exp (- c_1 (N - 4 \sqrt N)) - \exp (- 2 \,c_1 N)) \,P \,(\tau < \tau_2^+). \end{array}
\end{equation}
 Since $\sigma$ is an almost surely finite stopping time and the process $(Z_t)$ stopped at time $\sigma$ is a supermartingale
 according to Lemma~\ref{lem:outer-supermartingale}, the optimal stopping theorem implies that
\begin{equation}
\label{eq:outer-fast-2}
\begin{array}{rcl}
  E \,(Z_{\sigma} \,| \,X_0 + Y_0 \geq N) & \leq &
  E \,(Z_{\sigma} \,| \,X_0 + Y_0 = N) \vspace*{4pt} \\ & \leq &
  E \,(Z_0 \,| \,X_0 + Y_0 = N) \ = \ \exp (- c_1 N). \end{array}
\end{equation}
 Combining \eqref{eq:outer-fast-1}--\eqref{eq:outer-fast-2}, we get
\begin{equation}
\label{eq:outer-fast-3}
\begin{array}{rcl}
  P \,(\tau < \tau_2^+ \,| \,X_0 + Y_0 \geq N) & \leq &
    \displaystyle \frac{\exp (- c_1 N) - \exp (- 2 \,c_1 N)}{\exp (- c_1 (N - 4 \sqrt N)) - \exp (- 2 \,c_1 N)} \vspace*{8pt} \\ & \leq &
    \displaystyle \frac{\exp (- c_1 N)}{\exp (- c_1 (N - 4 \sqrt N))} \ = \ \exp (- 4 \,c_1 \sqrt N). \end{array}
\end{equation}
 Since, on the event that $\tau_2^+ < \tau$, the process stays in $\Omega_- \,\cup \,\Omega_+$ before it
 hits $\Omega_2$, using Lemma~\ref{lem:outer-sum} and standard large deviation estimates for the Poisson distribution,
 we also have
\begin{equation}
\label{eq:outer-fast-4}
  P \,(\tau_2^+ > N \ | \ \tau_2^+ < \tau \ \hbox{and} \ X_0 + Y_0 \geq N) \ \leq \ \exp (- c_4 N)
\end{equation}
 for some $c_4 > 0$ and all $N$ large.
 Finally, we combine \eqref{eq:outer-fast-3}--\eqref{eq:outer-fast-4} to conclude
 $$ \begin{array}{l}
     P \,(\tau_2^+ > N \ | \ X_0 + Y_0 \geq N) \vspace*{4pt} \\ \hspace*{25pt} = \
     P \,(\tau_2^+ > N \ | \ \tau < \tau_2^+ \ \hbox{and} \ X_0 + Y_0 \geq N) \,P \,(\tau < \tau_2^+ \ | \ X_0 + Y_0 \geq N) \vspace*{4pt} \\ \hspace*{50pt} + \
     P \,(\tau_2^+ > N \ | \ \tau_2^+ < \tau \ \hbox{and} \ X_0 + Y_0 \geq N) \,P \,(\tau_2^+ < \tau \ | \ X_0 + Y_0 \geq N) \vspace*{4pt} \\ \hspace*{25pt} \leq \
     P \,(\tau < \tau_2^+ \ | \ X_0 + Y_0 \geq N) + P \,(\tau_2^+ > N \ | \ \tau_2^+ < \tau \ \hbox{and} \ X_0 + Y_0 \geq N) \vspace*{4pt} \\ \hspace*{25pt} \leq \
       \exp (- 4 \,c_1 \sqrt N) + \exp (- c_4 N) \ \leq \ \exp (- c_3 \sqrt N) \end{array} $$
 for some $c_3 > 0$.
 This completes the proof.
\end{proof} \\


\noindent The second step of the proof is established in Lemma~\ref{lem:outer-stay} below.
 The proof again relies on Lemma~\ref{lem:outer-sum} but we also need the following preliminary result.
\begin{lemma} --
\label{lem:outer-difference}
 Let $b > 8$.
 For all $N$ large,
 $$ \begin{array}{l} \lim_{\,h \to 0} \,h^{-1} \,E \,((X_{t + h} - Y_{t + h}) - (X_t - Y_t) \ | \ (X_t, Y_t) \in \Omega_-) \ \leq \ 0. \end{array} $$
\end{lemma}
\begin{proof}
 This directly follows from the fact that the function
 $$ \psi (i, j) \ := \ N^{-2} \,(b/2) \ i^2 \,(N - j) $$
 is increasing with respect to $i$ and decreasing with respect to $j$,
 $$ \begin{array}{l}
    \lim_{\,h \to 0} \,h^{-1} \,E \,((X_{t + h} - Y_{t + h}) - (X_t - Y_t) \ | \ (X_t, Y_t) = (i, j)) \vspace*{4pt} \\ \hspace*{80pt} = \
      N^{-2} \,(b/2)(j^2 \,(N - i) - i^2 \,(N - j)) - (i - j) + O (1) \vspace*{4pt} \\ \hspace*{80pt} = \
    \psi (j, i) - \psi (i, j) - (i - j) + O (1) \end{array} $$
 and the fact that $i > j$ for all $(i, j) \in \Omega_-$.
\end{proof} \\ \\
 With Lemmas~\ref{lem:outer-sum} and \ref{lem:outer-difference} in hands, we are now ready to establish the second step of the proof,
 which can be more conveniently stated using the exit time
 $$ \tau_{\rho}^- \ := \ \inf \,\{t : (X_t, Y_t) \notin \Omega_{\rho} \}. $$
\begin{lemma} --
\label{lem:outer-stay}
 Let $b > 8$.
 There exists $c_5 > 0$ such that
 $$ P \,(\tau_0^- \leq 4 N \ | \ (X_0, Y_0) \in \Omega_1) \ \leq \ \exp (- c_5 \sqrt N) \quad \hbox{for all $N$ large}. $$
\end{lemma}
\begin{proof}
 From Lemmas~\ref{lem:outer-sum} and \ref{lem:outer-difference}, we deduce that
 $$ \begin{array}{l}
    \lim_{\,h \to 0} \,h^{-1} \,E \,(Y_{t + h} - Y_t \ | \ (X_t, Y_t) = (i, j)) \vspace*{4pt} \\ \hspace*{25pt} = \
     (1/2) \,\lim_{\,h \to 0} \,h^{-1} \,E \,((X_{t + h} + Y_{t + h}) - (X_t + Y_t) \ | \ (X_t, Y_t) = (i, j)) \vspace*{4pt} \\ \hspace*{50pt} - \
     (1/2) \,\lim_{\,h \to 0} \,h^{-1} \,E \,((X_{t + h} - Y_{t + h}) - (X_t - Y_t) \ | \ (X_t, Y_t) = (i, j)) \vspace*{4pt} \\ \hspace*{25pt} \geq \
     N (1/8)(b/8 - 1) \ > \ 0 \end{array} $$
 for all $(i, j) \in \Omega_-$.
 By symmetry,
 $$ \begin{array}{l} \lim_{\,h \to 0} \,h^{-1} \,E \,(X_{t + h} - X_t \ | \ (X_t, Y_t) = (i, j)) \ \geq \ N (1/8)(b/8 - 1) \ > \ 0 \end{array} $$
 for all $(i, j) \in \Omega_+$.
 In particular, letting
 $$ D_t \ := \ \dist ((X_t, Y_t), \Omega_2) \ = \ \min \,\{|X_t - u| \vee |Y_t - v| : (u, v) \in \Omega_2 \} $$
 there exists $c_6 > 0$ such that
\begin{equation}
\label{eq:outer-stay-1}
  \begin{array}{l} \lim_{\,h \to 0} \,h^{-1} \,E \,(D_{t + h} - D_t \ | \ (X_t, Y_t) = (i, j)) \ \geq \ c_6 N \end{array}
\end{equation}
 for all $(i, j) \in \Omega_0 \setminus \Omega_2$ therefore, there exists $c_7 > 0$ such that
 $$ \begin{array}{l} \lim_{\,h \to 0} \,h^{-1} \,E \,(\exp (c_7 D_{t + h}) - \exp (c_7 D_t) \ | \ (X_t, Y_t) = (i, j)) \ \leq \ 0 \quad \hbox{for all} \quad (i, j) \in \Omega_0 \setminus \Omega_2. \end{array} $$
 This follows from \eqref{eq:outer-stay-1} by using the same argument as in Lemma~\ref{lem:outer-supermartingale}.
 Applying the optimal stopping theorem at the stopping time $\sigma := \min (\tau_0^-, \tau_2^+)$, we deduce that
 $$ \begin{array}{l}
     E \,(\exp (c_7 D_0) \ | \ (X_0, Y_0) \in \partial \Omega_1) \ = \ \exp (c_7 \sqrt N) \vspace*{4pt} \\ \hspace*{25pt} \geq \
     E \,(\exp (c_7 D_{\tau_0^-})) \,P \,(\tau_0^- < \tau_2^+ \ | \ (X_0, Y_0) \in \partial \Omega_1) \vspace*{4pt} \\ \hspace*{50pt} + \
     E \,(\exp (c_7 D_{\tau_2^+})) \,(1 - P \,(\tau_0^- < \tau_2^+ \ | \ (X_0, Y_0) \in \partial \Omega_1)) \vspace*{4pt} \\ \hspace*{25pt} = \
     1 + (\exp (2 \,c_7 \sqrt N) - 1) \,P \,(\tau_0^- < \tau_2^+ \ | \ (X_0, Y_0) \in \partial \Omega_1) \end{array} $$
 from which it follows as in \eqref{eq:outer-fast-3} that
\begin{equation}
\label{eq:outer-stay-2}
   P \,(\tau_0^- < \tau_2^+ \ | \ (X_0, Y_0) \in \partial \Omega_1) \ \leq \ \exp (- c_7 \,\sqrt N).
\end{equation}
 To conclude, we let
 $$ u (\tau_0^-) \ := \ \# \ \hbox{upcrossings of $D_t$ across the interval $(0, \sqrt N)$ by time $\tau_0^-$}. $$
 The inequality in \eqref{eq:outer-stay-2} indicates that the number of upcrossings of the distance to $\Omega_2$ before the process leaves the larger
 square $\Omega_0$ is stochastically larger than a geometric random variable whose success probability tends to zero, which implies the existence
 of $c_8 > 0$ such that
 $$ \begin{array}{rcl}
      P \,(u (\tau_0^-) \leq N^2) & = & 1 - P \,(u (\tau_0^-) > N^2) \ \leq \ 1 - (1 - \exp (- c_7 \sqrt N))^{N^2} \vspace*{4pt} \\ & \leq &
      2 \,N^2 \exp (- c_7 \,\sqrt N) \ \leq \ \exp (- c_8 \sqrt N) \end{array} $$
 for all $N$ sufficiently large.
 Since in addition
\begin{enumerate}
 \item the number of jumps by time $4N$ is at most of the order of $N^2$ \vspace*{2pt}
 \item the number of jumps to perform more than $N^2$ upcrossings is at least $N^{5/2}$
\end{enumerate}
 the lemma follows from large deviation estimates for the Poisson distribution.
\end{proof} \\


\noindent Using the first two steps given by Lemmas~\ref{lem:outer-fast} and \ref{lem:outer-stay} respectively,  we now prove the third step:
 the first invasion is successful with probability close to one when $N$ is large.
\begin{lemma} --
\label{lem:outer-first-invasion}
 Let $b > 8$.
 There exists $c_9 > 0$ such that
 $$ P \,((X_t, Y_t) \notin \Omega_0 \ \hbox{for some} \ t \in (N, 4N) \ | \ X_0 + Y_0 \geq N) \ \leq \ \exp (- c_9 \sqrt N) $$
 for all $N$ sufficiently large.
\end{lemma}
\begin{proof}
 Applying Lemmas~\ref{lem:outer-fast} and \ref{lem:outer-stay} we obtain
 $$ \begin{array}{l}
     P \,((X_t, Y_t) \notin \Omega_0 \ \hbox{for some} \ t \in (N, 4N) \ | \ X_0 + Y_0 \geq N) \vspace*{4pt} \\ \hspace*{40pt} \leq \
     P \,(\tau_2^+ > N \ | \ X_0 + Y_0 \geq N) + P \,(\tau_0^- \leq 4 N \ | \ (X_0, Y_0) \in \Omega_1) \vspace*{4pt} \\ \hspace*{40pt} \leq \
       \exp (- c_3 \sqrt N) + \exp (- c_5 N) \ \leq \ \exp (- c_9 \sqrt N) \end{array} $$
 for all $N$ sufficiently large and a suitable constant $c_9 > 0$.
\end{proof} \\


\noindent The fourth step, which deals with subsequent invasions, is proved in Lemma \ref{lem:outer-next-invasion} below and can be deduced
 from the result of third step by also using the following preliminary lemma.
\begin{lemma} --
\label{lem:outer-up}
 Let $b > 8$.
 For all $N$ large,
 $$ \begin{array}{l} \lim_{\,h \to 0} \,h^{-1} \,E \,(Y_{t + h} - Y_t \ | \ X_t > N/2 \ \hbox{and} \ Y_t < N/2 + 2 \sqrt N) \ \geq \ N (1/4)(b/8 - 1). \end{array} $$
\end{lemma}
\begin{proof}
 Using again the monotonicity of $\psi$ introduced in Lemma \ref{lem:outer-difference}, we get
 $$ \begin{array}{l}
    \lim_{\,h \to 0} \,h^{-1} \,E \,(Y_{t + h} - Y_t \ | \ (X_t, Y_t) = (i, j)) \ = \ \psi (i, j) - j \vspace*{4pt} \\ \hspace*{50pt} \geq \
    \psi (N/2, N/2 + 2 \sqrt N) - (N/2 + 2 \sqrt N) \vspace*{4pt} \\ \hspace*{50pt} = \
      N^{-2} \,(b/2)(N/2)^2 \,(N - (N/2 + 2 \sqrt N)) - (N/2 + 2 \sqrt N) \vspace*{4pt} \\ \hspace*{50pt} = \
      N (1/2)(b/8 - 1) + O (\sqrt N) \ \geq \ N (1/4)(b/8 - 1) \end{array} $$
 for all $i > N/2$ and $j < N/2 + 2 \sqrt N$ and all $N$ sufficiently large.
\end{proof}
\begin{lemma} --
\label{lem:outer-next-invasion}
 Let $b > 8$.
 There exists $c_{10} > 0$ such that
 $$ P \,((X_t, Y_t) \notin \Omega_0 \ \hbox{for some} \ t \in (2N, 4N) \ | \ X_t > N/2 \ \hbox{for all} \ t \in (0, 2N)) \ \leq \ \exp (- c_{10} \sqrt N) $$
 for all $N$ sufficiently large.
\end{lemma}
\begin{proof}
 A direct application of Lemma~\ref{lem:outer-up} together with large deviation estimates for the Poisson random variable implies the
 existence of $c_{11} > 0$ such that
 $$ \begin{array}{l}
      P \,(X_t + Y_t < N \ \hbox{for all} \ t \leq N \ | \ X_t > N/2 \ \hbox{for all} \ t \in (0, 2N)) \vspace*{4pt} \\ \hspace*{25pt} \leq \
      P \,(Y_t < N/2 \ \hbox{for all} \ t \leq N \ | \ X_t > N/2 \ \hbox{for all} \ t \in (0, 2N)) \ \leq \ \exp (- c_{11} N) \end{array} $$
 for all $N$ sufficiently large.
 Then, applying Lemma~\ref{lem:outer-first-invasion}, we get
 $$ \begin{array}{l}
     P \,((X_t, Y_t) \notin \Omega_0 \ \hbox{for some} \ t \in (2N, 4N) \ | \ X_t > N/2 \ \hbox{for all} \ t \in (0, 2N)) \vspace*{4pt} \\ \hspace*{25pt} \leq \
     P \,((X_t, Y_t) \notin \Omega_0 \ \hbox{for some} \ t \in (2N, 4N) \ | \ X_N + Y_N \geq N) \vspace*{4pt} \\ \hspace*{50pt} + \
     P \,(X_t + Y_t < N \ \hbox{for all} \ t \leq N \ | \ X_t > N/2 \ \hbox{for all} \ t \in (0, 2N)) \vspace*{4pt} \\ \hspace*{25pt} \leq \
       \exp (- c_9 \sqrt N) + \exp (- c_{11} N) \ \leq \ \exp (- c_{10} \sqrt N) \end{array} $$
 for all $N$ large and a suitable constant $c_{10} > 0$.
\end{proof} \\


\noindent Finally, we deduce the first part of the theorem in the next lemma relying on Lemmas~\ref{lem:outer-first-invasion}
 and \ref{lem:outer-next-invasion} as well as a block construction to compare the process with oriented percolation.
\begin{lemma} --
\label{lem:outer-survival}
 Assume that $b > 8$ and $M = 1$. There exists $c > 0$ such that
 $$ P \,(\xi_t = \0 \ \hbox{for some} \ t > 0 \ | \ \xi_0 = \ind_x) \ \leq \ \exp (- c \sqrt N) \quad \hbox{for all $N$ sufficiently large}. $$
\end{lemma}
\begin{proof}
 The last ingredient is to compare the process properly rescaled in space and time with oriented site percolation on
 the directed graph $\mathcal H$ with vertex set
 $$ H \ := \ \{(z, n) \in \Z \times \Z_+ : z + n \ \hbox{is even} \} $$
 and in which there is an oriented edge
 $$ (z, n) \to (z', n') \quad \hbox{if and only if} \quad z' = z \pm 1 \ \ \hbox{and} \ \ n' = n + 1. $$
 See Durrett \cite{durrett_1984} for a definition and a review of oriented percolation.
 In order to couple the two processes, we declare site $(z, n) \in \mathcal H$ to be good whenever
 $$ (\xi_t (z), \xi_t (z + 1)) \in \Omega_0 \quad \hbox{for all} \quad t \in ((2n + 2) N, (2n + 4) N) $$
 and let $Z_n$ be the set of good sites at level $n$, i.e.,
 $$ Z_n \ := \ \{z \in \Z : (z, n) \in H \ \hbox{is good} \}. $$
 On the one hand, Lemma~\ref{lem:outer-first-invasion} implies that
\begin{equation}
\label{eq:outer-survival-1}
  \begin{array}{l}
    P \,(0 \notin Z_0 \ | \ \xi_0 (0) = N) \vspace*{4pt} \\ \hspace*{25pt} = \
    P \,((\xi_t (0), \xi_t (1)) \notin \Omega_0 \ \hbox{for some} \ t \in (2N, 4N) \ | \ \xi_0 (0) = N) \vspace*{4pt} \\ \hspace*{25pt} \leq \
    P \,((X_t, Y_t) \notin \Omega_0 \ \hbox{for some} \ t \in (N, 4N) \ | \ X_0 + Y_0 = N) \ \leq \ \exp (- c_9 \sqrt N). \end{array}
\end{equation}
 On the other hand, Lemma~\ref{lem:outer-next-invasion} shows that
\begin{equation}
\label{eq:outer-survival-2}
 \begin{array}{l}
   P \,(z \pm 1 \notin Z_{n + 1} \ | \ z \in Z_n) \vspace*{4pt} \\ \hspace*{25pt} = \
   P \,((\xi_t (z + 1), \xi_t (z + 2)) \notin \Omega_0 \ \hbox{for some} \ t \in ((2n + 4)N, (2n + 6)N) \ | \vspace*{4pt} \\ \hspace*{60pt}
        (\xi_t (z), \xi_t (z + 1)) \in \Omega_0 \ \hbox{for all} \ t \in ((2n + 2)N, (2n + 4)N)) \vspace*{4pt} \\ \hspace*{25pt} \leq \
   P \,((X_t, Y_t) \notin \Omega_0 \ \hbox{for some} \ t \in (2N, 4N) \ | \ X_t > N/2 \ \hbox{for all} \ t \in (0, 2N)) \vspace*{4pt} \\ \hspace*{25pt} \leq \
     \exp (- c_{10} \sqrt N). \end{array}
\end{equation}
 The inequality \eqref{eq:outer-survival-2} and Theorem 4.3 in \cite{durrett_1995} imply the existence of a coupling between the contact process
 with sexual reproduction and oriented percolation.
 More precisely, letting $W_n$ be the set of wet sites at level $n$ in a one dependent oriented site percolation process in which sites
 are closed with probability $\exp (- c_{10} \sqrt N)$, the processes can be coupled such that
\begin{equation}
\label{eq:outer-survival-3}
  W_n \,\subset \,Z_n \ \ \hbox{for all} \ n \in \Z_+ \quad \hbox{whenever} \quad W_0 \,\subset \,Z_0.
\end{equation}
 In other respects, taking $N$ large enough so that $1 - \exp (- c_{10} \sqrt N)$ is larger than the critical value of one dependent oriented
 site percolation, we have
\begin{equation}
\label{eq:outer-survival-4}
   P \,(W_n = \varnothing \ \hbox{for some} \ n \ | \ W_N = \{- N, \ldots, N \}) \ \leq \ \exp (- c_{12} N)
\end{equation}
 for a suitable constant $c_{12} > 0$.
 See Section 10 of \cite{durrett_1984} for a proof.
 Moreover,
\begin{equation}
\label{eq:outer-survival-5}
 \begin{array}{l}
   P \,(W_N \neq \{- N, \ldots, N \} \ | \ W_0 = \{0 \}) \vspace*{4pt} \\ \hspace*{25pt} = \
   P \,(W_n \neq \{- n, \ldots, n \} \ \hbox{for some} \ n \leq N \ | \ W_0 = \{0 \}) \vspace*{4pt} \\ \hspace*{50pt} \leq \
       (1/2)(N + 1)(N + 2) \,\exp (- c_{10} \sqrt N). \end{array}
\end{equation}
 Using the coupling \eqref{eq:outer-survival-3} and inequalities \eqref{eq:outer-survival-1}, \eqref{eq:outer-survival-4}
 and \eqref{eq:outer-survival-5}, we deduce that
 $$ \begin{array}{l}
      P \,(\xi_t = \0 \ \hbox{for some} \ t > 0 \ | \ \xi_0 (0) = N) \ \leq \
      P \,(Z_n = \varnothing \ \hbox{for some} \ n \ | \ Z_N = \{- N, \ldots, N \}) \vspace*{4pt} \\ \hspace*{50pt} + \
      P \,(Z_N \neq \{- N, \ldots, N \} \ | \ 0 \in Z_0) +
      P \,(0 \notin Z_0 \ | \ \xi_0 (0) = N) \vspace*{4pt} \\ \hspace*{25pt} \leq \
      P \,(W_n = \varnothing \ \hbox{for some} \ n \ | \ Z_N = \{- N, \ldots, N \}) \vspace*{4pt} \\ \hspace*{50pt} + \
      P \,(W_N \neq \{- N, \ldots, N \} \ | \ W_0 = \{0 \}) + \exp (- c_9 \sqrt N) \vspace*{4pt} \\ \hspace*{25pt} \leq \
        \exp (- c_{12} N) + (1/2)(N + 1)(N + 2) \,\exp (- c_{10} \sqrt N) + \exp (- c_9 \sqrt N) \ \leq \ \exp (- c \sqrt N) \end{array} $$
 for a suitable $c > 0$ and $N$ sufficiently large.
\end{proof}


\section{Proof of Theorem \ref{th:survival} under assumption A2}
\label{sec:inner-survival}

\indent The proof of the second part of the theorem relies on similar arguments but is shorter partly because first and subsequent
 invasions can be dealt with using the same proof.
 The proof is now divided into two steps, the first one focusing on patch $x$ and the second one on invasion $x \to y$. \vspace*{4pt} \\
\noindent {\bf Step 1} -- Even in the absence of migrations, starting with the population in patch $x$ above a certain threshold,
 the population quickly increases up to a larger threshold and then stays above the lower threshold for a long time.
 This step is similar to steps 1 and 2 in the previous section but focuses on the local population in one patch instead of the global
 population in two patches. \vspace*{4pt} \\
\noindent {\bf Step 2} -- In the presence of migrations and as long as the population in patch $x$ is above the lower threshold,
 the population in patch $y$ increases quickly up to the upper threshold and then stays above the lower threshold for an arbitrarily
 long time. \vspace*{4pt} \\
 The second part of the theorem is again obtained by combining these two steps with a block construction in order to compare the
 process properly rescaled in space and time with oriented percolation.
 Recall that $Q \,(X) = a X^2 (1 - X) - X$ with $a > 4$ and denote by
 $$ c_- \ := \ c_- (a) \ = \ \frac{1}{2} - \sqrt{\frac{1}{4} - \frac{1}{a}} \qquad \hbox{and} \qquad
    c_+ \ := \ c_+ (a) \ = \ \frac{1}{2} + \sqrt{\frac{1}{4} - \frac{1}{a}} $$
 its two nontrivial roots.
 Step~1 is established in Lemmas \ref{lem:inner-fast}--\ref{lem:inner-stay} below.
 Both of these lemmas rely on the next preliminary result which can be seen as the analog of Lemma \ref{lem:outer-supermartingale}.


\begin{lemma} --
\label{lem:inner-supermartingale}
 Let $a > 4$.
 There exists $c_{13} > 0$ such that
 $$ \begin{array}{l} \lim_{\,h \to 0} \,h^{-1} \,E \,(Z_{t + h} - Z_t \ | \ X_t \in c_+ N + (- 4 \sqrt N, - \sqrt N)) \ \leq \ 0 \quad \hbox{for all $N$ large} \end{array} $$
 where $Z_t := \exp \,(- c_{13} X_t)$.
\end{lemma}
\begin{proof}
 Since $Q' (c_+) < 0$, for all $i \in c_+ N + (- 4 \sqrt N, - \sqrt N)$,
\begin{equation}
\label{eq:inner-1}
  \begin{array}{l}
    \lim_{\,h \to 0} \,h^{-1} \,E \,(X_{t + h} - X_t \ | \ X_t = i) \ \geq \
     N^{-2} \,a \,i^2 \,(N - i) - i + O (1) \vspace*{4pt} \\ \hspace*{25pt} \geq \
     N \,\min_{\,c \,\in (1, 4)} \ (a \,(c_+ - c / \sqrt N)^2 (1 - c_+ + c / \sqrt N) - (c_+ - c / \sqrt N)) + O (1) \vspace*{4pt} \\ \hspace*{25pt} = \
     N \,\min_{\,c \,\in (1, 4)} \ Q (c_+ - c / \sqrt N) + O (1) \vspace*{4pt} \\ \hspace*{25pt} = \
     N \,\min_{\,c \,\in (1, 4)} \ (Q (c_+) - (c / \sqrt N) \,Q' (c_+)) + O (1) \vspace*{4pt} \\ \hspace*{25pt} \geq \
     - \sqrt N \ Q' (c_+) + O (1) \ \geq \ - (1/2) \sqrt N \ Q' (c_+) \ > \ 0 \end{array}
\end{equation}
 for all $N$ sufficiently large.
 To complete the proof of the lemma, we repeat the same argument as in the proof of Lemma~\ref{lem:outer-supermartingale}, namely we define the functions
 $$ \begin{array}{l}
     \phi_i (c) \ := \ \lim_{\,h \to 0} \,h^{-1} \,E \,(\exp \,(- c \,X_{t + h}) - \exp \,(- c \,X_t) \ | \ X_t = i) \end{array} $$
 and notice that, according to \eqref{eq:inner-1}, the derivative satisfies
 $$ \begin{array}{rcl}
     \phi_i' (c) & = & \lim_{\,h \to 0} \,h^{-1} \,E \,(- X_{t + h} \,\exp \,(- c \,X_{t + h}) + X_t \,\exp \,(- c \,X_t) \ | \ X_t = i) \vspace*{4pt} \\
     \phi_i' (0) & \leq & - c_{14} \ := \ (1/2) \sqrt N \ Q' (c_+) \quad \hbox{for all} \quad i \in c_+ N + (- 4 \sqrt N, - \sqrt N). \end{array} $$
 Using also that $\phi_i (0) = 0$, we get the existence of $c_{13} > 0$ such that
 $$ \begin{array}{l}
     \phi_i (c_{13}) \ = \ \lim_{\,h \to 0} \,h^{-1} \,E \,(\exp \,(- c_{13} X_{t + h}) - \exp \,(- c_{13} X_t) \ | \ X_t = i) \ \leq \ 0 \end{array} $$
 for all $i \in c_+ N + (- 4 \sqrt N, - \sqrt N)$.
 This completes the proof.
\end{proof} \\ \\
 To state our next two lemmas, we introduce, for all $\rho > 0$, the stopping times
 $$ \tau_{\rho}^+ \ := \ \inf \,\{t : X_t > c_+ N - \rho \,\sqrt N \} \quad \hbox{and} \quad
    \tau_{\rho}^- \ := \ \inf \,\{t : X_t < c_+ N - \rho \,\sqrt N \}. $$
\begin{lemma} --
\label{lem:inner-fast}
 Let $a > 4$.
 There exists $c_{15} > 0$ such that
 $$ P \,(\tau_2^+ > N \ | \ X_0 > c_+ N - 3 \sqrt N) \ \leq \ \exp (- c_{15} \sqrt N) \quad \hbox{for all $N$ large}. $$
\end{lemma}
\begin{proof}
 Introducing the stopping time $\sigma := \min (\tau_2^+, \tau_4^-)$, we have
\begin{equation}
\label{eq:inner-fast-1}
  \begin{array}{rcl}
    E \,(Z_{\sigma}) & = &
    E \,(Z_{\tau_2^+}) \,P \,(\tau_2^+ < \tau_4^-) + E \,(Z_{\tau_4^-}) \,P \,(\tau_2^+ > \tau_4^-) \vspace*{4pt} \\ & = &
      \exp (- c_{13} \,(c_+ N - 2 \sqrt N)) \,P \,(\tau_2^+ < \tau_4^-) \vspace*{4pt} \\ && \hspace*{10pt} + \
      \exp (- c_{13} N (c_+ - 4 / \sqrt N)) \,P \,(\tau_2^+ > \tau_4^-) \vspace*{4pt} \\ & = &
      \exp (- c_{13} \,(c_+ N - 2 \sqrt N)) \vspace*{4pt} \\ && \hspace*{10pt} + \
     (\exp (- c_{13} \,(c_+ N - 4 \sqrt N)) - \exp (- c_{13} \,(c_+ N - 2 \sqrt N))) \,P \,(\tau_2^+ > \tau_4^-). \end{array}
\end{equation}
 Since time $\sigma$ is almost surely finite and the process $(Z_t)$ stopped at time $\sigma$ is a supermartingale
 according to Lemma~\ref{lem:inner-supermartingale}, the optimal stopping theorem implies that
\begin{equation}
\label{eq:inner-fast-2}
\begin{array}{l}
  E \,(Z_{\sigma} \,| \,X_0 > c_+ N - 3 \sqrt N) \ \leq \
  E \,(Z_{\sigma} \,| \,X_0 = c_+ N - 3 \sqrt N) \vspace*{4pt} \\ \hspace*{50pt} \leq \
  E \,(Z_0 \,| \,X_0 = c_+ N - 3 \sqrt N) \ = \ \exp (- c_{13} \,(c_+ N - 3 \sqrt N)). \end{array}
\end{equation}
 Combining \eqref{eq:inner-fast-1}--\eqref{eq:inner-fast-2}, we get
\begin{equation}
\label{eq:inner-fast-3}
\begin{array}{rcl}
  P \,(\tau_2^+ > \tau_4^- \,| \,X_0 > c_+ N - 3 \sqrt N) & \leq &
      (\exp (c_{13} \sqrt N) - 1)(\exp (2 \,c_{13} \sqrt N) - 1)^{-1} \vspace*{4pt} \\ & \leq &
       \exp (- c_{13} \sqrt N). \end{array}
\end{equation}
 Moreover, by definition of the stopping time $\sigma$,
 $$ X_t \in c_+ N + (- 4 \sqrt N, - 2 \sqrt N) \quad \hbox{for all} \quad t < \sigma $$
 therefore inequality \eqref{eq:inner-1} in the previous lemma and large deviation estimates for the Binomial and the
 Poisson distributions also imply that
\begin{equation}
\label{eq:inner-fast-4}
  P \,(\tau_2^+ > N \ | \ \tau_2^+ < \tau_4^- \ \hbox{and} \ X_0 > c_+ N - 3 \sqrt N) \ \leq \ \exp (- c_{16} N)
\end{equation}
 for some $c_{16} > 0$ and all $N$ large.
 Combining \eqref{eq:inner-fast-3}--\eqref{eq:inner-fast-4}, we conclude
 $$ \begin{array}{l}
     P \,(\tau_2^+ > N \ | \ X_0 > c_+ N - 3 \sqrt N) \vspace*{4pt} \\ \hspace*{25pt} \leq \
     P \,(\tau_2^+ > \tau_4^- \ | \ X_0 > c_+ N - 3 \sqrt N) + P \,(\tau_2^+ > N \ | \ \tau_2^+ < \tau_4^- \ \hbox{and} \ X_0 > c_+ N - 3 \sqrt N) \vspace*{4pt} \\ \hspace*{25pt} \leq \
       \exp (- c_{13} \sqrt N) + \exp (- c_{16} N) \ \leq \ \exp (- c_{15} \sqrt N) \end{array} $$
 for some $c_{15} > 0$.
 This completes the proof.
\end{proof}
\begin{lemma} --
\label{lem:inner-stay}
 Let $a > 4$.
 There exists $c_{17} > 0$ such that
 $$ P \,(\tau_3^- \leq 4 N \ | \ X_0 > c_+ N - 2 \sqrt N) \ \leq \ \exp (- c_{17} \sqrt N) \quad \hbox{for all $N$ large}. $$
\end{lemma}
\begin{proof}
 Repeating the proof of Lemma~\ref{lem:inner-fast} but using the stopping time $\min (\tau_1^+, \tau_3^-)$ instead of the stopping
 time $\sigma$, it can be proved that
\begin{equation}
\label{eq:inner-stay-1}
  P \,(\tau_1^+ > \tau_3^- \,| \,X_0 > c_+ N - 2 \sqrt N) \ \leq \ \exp (- c_{13} \sqrt N) \quad \hbox{for all $N$ large},
\end{equation}
 which is the analog of \eqref{eq:inner-fast-3} above.
 To conclude, we let
 $$ u (\tau_3^-) \ := \ \# \ \hbox{upcrossings of $X_t$ across the interval $c_+ N + (- 2 \sqrt N, - \sqrt N)$ by time $\tau_3^-$} $$
 and follow the same idea as in the proof of Lemma~\ref{lem:outer-stay}.
 It follows from \eqref{eq:inner-stay-1} that the number of upcrossings is stochastically larger than a geometric random variable
 whose success probability tends to zero: there exists a constant $c_{18} > 0$ such that
 $$ \begin{array}{rcl}
      P \,(u (\tau_3^-) \leq N^2) & = & 1 - P \,(u (\tau_3^-) > N^2) \ \leq \ 1 - (1 - \exp (- c_{13} \sqrt N))^{N^2} \vspace*{4pt} \\ & \leq &
      2 \,N^2 \exp (- c_{13} \sqrt N) \ \leq \ \exp (- c_{18} \sqrt N) \end{array} $$
 for all $N$ sufficiently large.
 Since in addition
\begin{enumerate}
 \item the number of jumps by time $4N$ is at most of the order of $N^2$ \vspace*{2pt}
 \item the number of jumps to perform more than $N^2$ upcrossings is at least $N^{5/2}$
\end{enumerate}
 the lemma follows from large deviation estimates for the Poisson distribution.
\end{proof} \\


\noindent Note that Lemma \ref{lem:inner-stay} applied to patch $y$ rather than patch $x$ already gives part of the second step of the proof.
 For all $\rho > 0$, we introduce the stopping times
 $$ \sigma_{\rho}^+ \ := \ \inf \,\{t : Y_t > c_+ N - \rho \,\sqrt N \} \quad \hbox{and} \quad
    \sigma_{\rho}^- \ := \ \inf \,\{t : Y_t < c_+ N - \rho \,\sqrt N \}. $$
 To complete the proof of step~2, it suffices to establish the following lemma.
\begin{lemma} --
\label{lem:inner-drift}
 Let $a > 4$ and $b > 2 a^3 \,c_-^4$.
 There exists $c_{19} > 0$ such that
 $$ P \,(\sigma_2^+ > 4N \ | \ X_t > c_+ N - 3 \sqrt N \ \hbox{for all} \ t \in (N, 4N)) \ \leq \ \exp (- c_{19} N) \quad \hbox{for all $N$ large}. $$
\end{lemma}
\begin{proof}
 Observing that $c_+ + c_- = 1$ and $c_+ \,c_- = a^{-1}$, we deduce
 $$ b \ > \ 2 a^3 \,c_-^4 \ = \ 2 \,(a c_-)^3 \,c_- \ = \ 2 c_+^{-3} \,c_- \ = \ 2 c_+^{-2} \,c_- \,(1 - c_-)^{-1}. $$
 In particular, for all $N$ large,
 $$ b \ > \ 2 \,(c_+ - 3 / \sqrt N)^{-2} \,(c_- + 2 / \sqrt N) (1 - c_- - 2 / \sqrt N)^{-1} $$
 from which it follows that, for all $j < c_- N + 2 \sqrt N$,
\begin{equation}
\label{eq:inner-drift-1}
  \begin{array}{l}
    \lim_{\,h \to 0} \,h^{-1} \,E \,(Y_{t + h} - Y_t \ | \ X_t > c_+ N - 3 \sqrt N \ \hbox{and} \ Y_t = j) \vspace*{4pt} \\ \hspace*{25pt} \geq \
     (b/2)(c_+ - 3 / \sqrt N)^2 \,(N - j) - j + O (1) \vspace*{4pt} \\ \hspace*{25pt} \geq \
     (b/2)(c_+ - 3 / \sqrt N)^2 \,(N - c_- N - 2 \sqrt N) - c_- N - 2 \sqrt N + O (1) \ \geq \ c_{20} \,N \end{array}
\end{equation}
 for some constant $c_{20} > 0$.
 In addition, for all $j \in (c_- N + 2 \sqrt N, c_+ N - 2 \sqrt N)$,
\begin{equation}
\label{eq:inner-drift-2}
  \begin{array}{l}
    \lim_{\,h \to 0} \,h^{-1} \,E \,(Y_{t + h} - Y_t \ | \ Y_t = j) \ \geq \
     N^{-2} \,a \,j^2 \,(N - j) - j + O (1) \vspace*{4pt} \\ \hspace*{25pt} \geq \
     N \min \,(Q (c_- + 2 / \sqrt N), Q (c_+ - 2 / \sqrt N)) + O (1) \vspace*{4pt} \\ \hspace*{25pt} = \
     N \min \ (Q (c_-) + (2 / \sqrt N) \,Q' (c_-), Q (c_+) - (2 / \sqrt N) \,Q' (c_+)) + O (1) \vspace*{4pt} \\ \hspace*{25pt} = \
     2 \sqrt N \ \min (Q' (c_-), - Q' (c_+)) + O (1) \ \geq \ \sqrt N \ \min (Q' (c_-), - Q' (c_+)) \ > \ 0 \end{array}
\end{equation}
 for all $N$ sufficiently large.
 Combining \eqref{eq:inner-drift-1}--\eqref{eq:inner-drift-2} together with large deviation estimates for the Binomial distribution and the
 Poisson distribution implies the result.
\end{proof} \\


\noindent The second part of the theorem is proved in the next lemma relying on Lemmas~\ref{lem:inner-fast}--\ref{lem:inner-drift}
 as well as a block construction to compare the process with oriented percolation.
\begin{lemma} --
\label{lem:inner-survival}
 Let $a > 4$ and $b > 2 a^3 \,c_-^4$.
 There exists $c > 0$ such that
 $$ P \,(\xi_t = \0 \ \hbox{for some} \ t > 0 \ | \ \xi_0 = \ind_x) \ \leq \ \exp (- c \sqrt N) \quad \hbox{for all $N$ sufficiently large}. $$
\end{lemma}
\begin{proof}
 This is similar to the proof of Lemma~\ref{lem:outer-survival}.
 Having the same directed graph $\mathcal H$ as in the previous section, we now declare site $(z, n) \in \mathcal H$ to be good whenever
 $$ \xi_t (z) \ > \ c_+ N - 3 \sqrt N \quad \hbox{at time} \quad t = 4n N $$
 and again define the set of good sites at level $n$ by setting
 $$ Z_n \ := \ \{z \in \Z : (z, n) \in H \ \hbox{is good} \}. $$
 First, we combine Lemmas~\ref{lem:inner-fast} and \ref{lem:inner-stay} to get
\begin{equation}
\label{eq:inner-survival-1}
  \begin{array}{l}
    P \,(\xi_t (z) \leq c_+ N - 3 \sqrt N \ \hbox{for some} \ t \in 4n N + (N, 4N) \ | \ z \in Z_n) \vspace*{4pt} \\ \hspace*{25pt} \leq \
    P \,(X_t \leq c_+ N - 3 \sqrt N \ \hbox{for some} \ t \in (N, 4N) \ | \ X_0 > c_+ N - 3 \sqrt N) \vspace*{4pt} \\ \hspace*{25pt} \leq \
    P \,(\tau_2^+ > N \ | \ X_0 > c_+ N - 3 \sqrt N) + P \,(\tau_3^- \leq 4 N \ | \ X_0 > c_+ N - 2 \sqrt N) \vspace*{4pt} \\ \hspace*{25pt} \leq \
      \exp (- c_{15} \sqrt N) + \exp (- c_{17} \sqrt N) \end{array}
\end{equation}
 for all $N$ large.
 Using Lemma~\ref{lem:inner-drift} and again Lemma~\ref{lem:inner-stay}, we also get
\begin{equation}
\label{eq:inner-survival-2}
 \begin{array}{l}
   P \,(Y_{4N} \leq c_+ N - 3 \sqrt N \ | \ X_t > c_+ N - 3 \sqrt N \ \hbox{for all} \ t \in (N, 4N)) \vspace*{4pt} \\ \hspace*{50pt} \leq \
   P \,(\sigma_2^+ > 4N \ | \ X_t > c_+ N - 3 \sqrt N \ \hbox{for all} \ t \in (N, 4N)) \vspace*{4pt} \\ \hspace*{80pt} + \
   P \,(Y_{4N} \leq c_+ N - 3 \sqrt N \ | \ \sigma_2^+ \leq 4N) \vspace*{4pt} \\ \hspace*{50pt} \leq \
     \exp (- c_{19} N) + P \,(\sigma_3^- \leq 4N \ | \ Y_0 > c_+ N - 2 \sqrt N) \vspace*{4pt} \\ \hspace*{50pt} \leq \
     \exp (- c_{19} N) + \exp (- c_{17} \sqrt N) \end{array}
\end{equation}
 for all $N$ large.
 Combining \eqref{eq:inner-survival-1}--\eqref{eq:inner-survival-2}, we deduce
\begin{equation}
\label{eq:inner-survival-3}
 \begin{array}{l}
   P \,(z \pm 1 \notin Z_{n + 1} \ | \ z \in Z_n) \vspace*{4pt} \\ \hspace*{25pt} \leq \
   P \,(\xi_t (z) \leq c_+ N - 3 \sqrt N \ \hbox{for some} \ t \in 4n N + (N, 4N) \ | \ z \in Z_n) \vspace*{4pt} \\ \hspace*{40pt} + \
   P \,(z \pm 1 \notin Z_{n + 1} \ \hbox{and} \ \xi_t (z) > c_+ N - 3 \sqrt N \ \hbox{for all} \ t \in 4n N + (N, 4N)) \vspace*{4pt} \\ \hspace*{25pt} \leq \
     \exp (- c_{15} \sqrt N) + \exp (- c_{17} \sqrt N) \vspace*{4pt} \\ \hspace*{40pt} + \
   P \,(Y_{4N} \leq c_+ N - 3 \sqrt N \ | \ X_t > c_+ N - 3 \sqrt N \ \hbox{for all} \ t \in (N, 4N)) \vspace*{4pt} \\ \hspace*{25pt} \leq \
     \exp (- c_{15} \sqrt N) + 2 \,\exp (- c_{17} \sqrt N) + \exp (- c_{19} N) \ \leq \ \exp (- c_{21} \sqrt N) \end{array}
\end{equation}
 for some constant $c_{21} > 0$ and all $N$ large.
 Repeating the same arguments as in the proof of Lemma~\ref{lem:outer-survival} but using inequality \eqref{eq:inner-survival-3}
 instead of inequality \eqref{eq:outer-survival-2}, we conclude that
 $$ \begin{array}{l}
      P \,(\xi_t \ \hbox{dies out} \ | \ \xi_0 (0) = N) \vspace*{4pt} \\ \hspace*{25pt} \leq \
        \exp (- c_{12} N) + (1/2)(N + 1)(N + 2) \,\exp (- c_{21} \sqrt N) \ \leq \ \exp (- c \sqrt N) \end{array} $$
 for a suitable $c > 0$ and all $N$ large.
\end{proof}


\section{Proof of Theorem \ref{th:range}}
\label{sec:range-extinction}

\indent This section is devoted to the proof of Theorem~\ref{th:range} which shows in particular that, for any choice of the
 parameters $a$, $b$ and $N$, the probability that the process starting with a single occupied patch survives tends to zero as
 the dispersal range tends to infinity.
 This together with Theorem~\ref{th:survival} suggests that dispersal promotes extinction of processes with sexual reproduction,
 whereas dispersal is known to promote survival of the basic contact process with no sexual reproduction.
 The proof relies on the following two key ingredients:
\begin{enumerate}
 \item In the absence of migrations: $b = 0$, the process starting with a single fully occupied patch goes extinct almost
  surely in a finite time.
  This directly follows from the fact that, in this case, the process is irreducible and has a finite state space. \vspace*{4pt}
 \item Calling a {\bf collision} the event that two offspring produced at the source patch $x$ are sent to the same patch
  $y \neq x$, in the presence of migrations: $b \neq 0$, but in the absence of collisions, offspring sent outside the source
  patch cannot reproduce due to the birth mechanism therefore, in view also of the previous point, the process dies out.
\end{enumerate}
 In particular, to find an upper bound for the survival probability, it suffices to find an upper bound for the probability
 of a collision since we have
\begin{equation}
\label{eq:range-1}
  P \,(\xi_t \neq \0 \ \hbox{for all} \ t > 0 \ | \ \xi_0 = \ind_x) \ \leq \ P \,(\hbox{collision} \ | \ \xi_0 = \ind_x).
\end{equation}
 In addition, since the probability of a collision is related to the number of individuals produced at patch~$x$ and sent
 outside the patch which, in turn, is related to the number of individuals at patch~$x$ in the absence of migrations, to find
 an upper bound for the probability of a collision, the first step is to find an upper bound for the time spent in state~$j$
 defined as
 $$ \tau_j \ := \int_0^{\infty} P \,(X_t = j \,| \,X_0 = N) \,dt \quad \hbox{for} \quad j = 1, 2, \ldots, N. $$
 More precisely, we have the following upper bound.
\begin{lemma} --
\label{lem:range-extinction}
 Let $b = 0$. Then,
 $$ \sum_{j = 1}^N \ j\,\tau_j \ \leq \ \sum_{j = 1}^N \ \sum_{i = 0}^j \ (a/4)^i. $$
\end{lemma}
\begin{proof}
 To simplify some tedious calculations and find the upper bound in the statement of the lemma, we first observe that the number
 of individuals at patch $x$ is dominated by the number of particles in a certain simple birth and death process truncated at state $N$.
 More precisely, note that the rate of transition $j \to j + 1$ is bounded by
 $$ a \ \frac{j \,(j - 1)(N - j)}{N \,(N - 1)} \ \leq \ a \ \frac{j^2 \,(N - j)}{N^2} \ \leq \ (a/4) \,j. $$
 In particular, standard coupling arguments imply that
\begin{equation}
\label{eq:range-extinction-1}
  P \,(X_t \geq i \ | \ X_0 = N) \ \leq \ P \,(Z_t \geq i \ | \ Z_0 = N) \quad \hbox{for all} \quad i = 1, 2, \ldots, N
\end{equation}
 where $Z_t$ is the process with transitions
 $$ \begin{array}{rclcll}
     j & \to & j + 1 & \ \ \hbox{at rate} \ \ & \beta_j := (a/4) \,j & \ \ \hbox{for} \ j = 0, 1, \ldots, N - 1 \vspace*{2pt} \\
     j & \to & j - 1 & \ \ \hbox{at rate} \ \ &   \mu_j :=         j & \ \ \hbox{for} \ j = 0, 1, \ldots, N. \end{array} $$
 In particular, letting $\sigma_j$ denote the amount of time the process $Z_t$ spends in state $j$, which can be seen as the analog
 of the amount of time $\tau_j$, inequality \eqref{eq:range-extinction-1} implies that
\begin{equation}
\label{eq:range-extinction-2}
 \begin{array}{rcl}
   \displaystyle \sum_{j = 1}^N \ j \,\tau_j & = &
   \displaystyle \int_0^{\infty} \ \sum_{j = 1}^N \ j \,P \,(X_t = j) \,dt \ = \
   \displaystyle \int_0^{\infty} \ \sum_{j = 1}^N \ \sum_{i = 1}^j \ P \,(X_t = j) \,dt \vspace*{4pt} \\ & = &
   \displaystyle \int_0^{\infty} \ \sum_{i = 1}^N \ \sum_{j = i}^N \ P \,(X_t = j) \,dt \ = \
   \displaystyle \int_0^{\infty} \ \sum_{i = 1}^N \ P \,(X_t \geq i) \,dt \vspace*{4pt} \\ & \leq &
   \displaystyle \int_0^{\infty} \ \sum_{i = 1}^N \ P \,(Z_t \geq i) \,dt \ = \
   \displaystyle \sum_{j = 1}^N \ j \,\sigma_j. \end{array}
\end{equation}
 Now, to find an upper bound for the occupation times $\sigma_j$, we first let $v_j$ denote the expected number of visits of $Z_t$ in
 state $j$, which gives the recursive relationship
\begin{equation}
\label{eq:range-extinction-3}
 \begin{array}{rcl}
   v_j & = & \mu_{j + 1} \,(\beta_{j + 1} + \mu_{j + 1})^{-1} \,v_{j + 1} + \beta_{j - 1} \,(\beta_{j - 1} + \mu_{j - 1})^{-1} \,v_{j - 1} \vspace*{4pt} \\
       & = & 4 \,(a + 4)^{-1} \,v_{j + 1} + a \,(a + 4)^{-1} \,v_{j - 1} \end{array}
\end{equation}
 for $j = 1, 2, \ldots, N - 1$, with boundary conditions
\begin{equation}
\label{eq:range-extinction-4}
   v_0 \ = \ \mu_1 \,(\beta_1 + \mu_1)^{-1} \,v_1 \quad \hbox{and} \quad
   v_N \ = \ 1 + \beta_{N - 1} \,(\beta_{N - 1} + \mu_{N - 1})^{-1} \,v_{N - 1}.
\end{equation}
 Note that the extra 1 in the expression of $v_N$ comes from the fact that $Z_0 = N$.
 We observe that the recursive relationship \eqref{eq:range-extinction-3} can be re-written as
 $$ v_j \ = \ (1 + a/4) \,v_{j - 1} - (a/4) \,v_{j - 2} \quad \hbox{for} \quad j = 2, 3, \ldots, N $$
 which has characteristic polynomial
 $$ X^2 - (1 + a/4) X + a/4 \ = \ (X - 1)(X - a/4). $$
 Using in addition that $v_0 = 1$ since state 0 is absorbing, we get
\begin{equation}
\label{eq:range-extinction-5}
  v_1 \ = \ 1 + (a/4) \quad \hbox{and} \quad v_j \ = \ c_{22} + c_{23} \,(a/4)^j \ = \ 1 + (a/4) + \cdots + (a/4)^j
\end{equation}
 from which it follows that
\begin{equation}
\label{eq:range-extinction-6}
  \sigma_j \ = \ \frac{v_j}{\beta_j + \mu_j}
           \ = \ j^{-1} \,(1 + a/4)^{-1} \ \sum_{i = 0}^j \ (a/4)^i
           \ \leq \ j^{-1} \ \sum_{i = 0}^j \ (a/4)^i
\end{equation}
 for $j = 1, 2, \ldots, N - 1$.
 Using \eqref{eq:range-extinction-4}--\eqref{eq:range-extinction-5}, we also deduce
\begin{equation}
\label{eq:range-extinction-7}
 \begin{array}{rcl}
    \sigma_N & = & \displaystyle \frac{1}{\mu_N} \ \bigg(1 + \frac{\beta_{N - 1}}{\beta_{N - 1} + \mu_{N - 1}} \ v_{N - 1} \bigg)
             \ \leq \ \displaystyle N^{-1} \ \bigg(1 + \frac{a/4}{1 + a/4} \ \sum_{i = 0}^{N - 1} \ (a/4)^i \bigg) \vspace*{4pt} \\
             & \leq & \displaystyle N^{-1} \ \bigg(1 + (a/4) \ \sum_{i = 0}^{N - 1} \ (a/4)^i \bigg)
             \ = \ N^{-1} \ \sum_{i = 0}^N \ (a/4)^i. \end{array}
\end{equation}
 Finally, combining \eqref{eq:range-extinction-2} with \eqref{eq:range-extinction-6}--\eqref{eq:range-extinction-7}, we get
 $$ \sum_{j = 1}^N \ j \,\tau_j \ \leq \
    \sum_{j = 1}^N \ j \,\sigma_j \ \leq \
    \sum_{j = 1}^N \ j \,j^{-1} \ \sum_{i = 0}^j \ (a/4)^i \ = \
    \sum_{j = 1}^N \ \sum_{i = 0}^j \ (a/4)^i $$
 which completes the proof.
\end{proof} \\ \\
 Using Lemma \ref{lem:range-extinction}, we now deduce upper bounds for the probability of a collision.
\begin{lemma} --
\label{lem:range-collision}
 For all $M$ large
\begin{enumerate}
 \item and all $a < 4$, we have \vspace*{4pt}
 \item[] \hspace*{20pt} $P \,(\hbox{collision} \ | \ \xi_0 = \ind_x) \ \leq \ M^{-1/3} \,(1/2 + b \,N (1 - a/4)^{-1})$. \vspace*{4pt}
 \item and $a = 4$, we have \vspace*{4pt}
 \item[] \hspace*{20pt} $P \,(\hbox{collision} \ | \ \xi_0 = \ind_x) \ \leq \ M^{-1/3} \,(1/2 + (b/2)(N + 2)^2)$. \vspace*{4pt}
 \item and all $a > 4$, we have \vspace*{4pt}
 \item[] \hspace*{20pt} $P \,(\hbox{collision} \ | \ \xi_0 = \ind_x) \ \leq \ M^{-1/3} \,(1/2 + b \,(a/4 - 1)^{-2} \,(a/4)^{N + 2})$.
\end{enumerate}
\end{lemma}
\begin{proof}
 Let $X$ denote the number of individuals born at patch $x$ and then sent outside the patch before the patch goes extinct.
 The idea is to use the following conditioning:
\begin{equation}
\label{eq:range-collision-1}
 \begin{array}{rcl}
   P \,(\hbox{collision}) & = &
   P \,(\hbox{collision} \ | \ X \leq M^{1/3}) \,P \,(X \leq M^{1/3}) \vspace*{4pt} \\ && \hspace*{50pt} + \
   P \,(\hbox{collision} \ | \ X > M^{1/3}) \,P \,(X > M^{1/3}) \vspace*{4pt} \\ & \leq &
   P \,(\hbox{collision} \ | \ X \leq M^{1/3}) + P \,(X > M^{1/3}). \end{array}
\end{equation}
 To estimate the first term in \eqref{eq:range-collision-1}, we observe that, since offspring sent outside the patch land on a patch chosen
 uniformly at random from a set of $2M$ patches, we have
\begin{equation}
\label{eq:range-collision-2}
  \begin{array}{l}
   \displaystyle P \,(\hbox{collision} \ | \ X \leq M^{1/3}) \ \leq \
   \displaystyle 1 - \prod_{j = 0}^{M^{1/3} - 1} \bigg(1 - \frac{j}{2M} \bigg) \vspace*{0pt} \\ \hspace*{35pt} \leq \
   \displaystyle 1 - \bigg(1 - \frac{M^{1/3}}{2M} \bigg)^{M^{1/3}} \ \leq \
   \displaystyle 1 - \exp \bigg(- \frac{M^{2/3}}{2M} \bigg) \ \leq \ (1/2) \,M^{-1/3} \end{array}
\end{equation}
 for all $M$ large.
 For the second term, we first use Lemma~\ref{lem:range-extinction} to get
 $$ \begin{array}{rcl}
     E X & = &    \displaystyle b \,N \ \sum_{j = 2}^N \ \frac{j \,(j - 1)}{N \,(N - 1)} \ \int_0^{\infty} P \,(X_t = j \ | \ X_0 = N) \,dt \vspace*{4pt} \\
         & \leq & \displaystyle b \ \sum_{j = 2}^N \ j \,\int_0^{\infty} P \,(X_t = j \ | \ X_0 = N) \,dt
         \ \leq \ \displaystyle b \ \sum_{j = 1}^N \ j \,\tau_j
         \ \leq \ \displaystyle b \ \sum_{j = 1}^N \ \sum_{i = 0}^j \ (a/4)^i. \end{array} $$
 In particular,
\begin{enumerate}
 \item for all $a < 4$, we have \vspace*{4pt}
 \item[] \hspace*{20pt} $E X \ \leq \ b \,(1 - a/4)^{-1} \ \sum_{j = 1}^N \ (1 - (a/4)^{j + 1}) \ \leq \ b \,N (1 - a/4)^{-1}$. \vspace*{8pt}
 \item for $a = 4$, we have \vspace*{4pt}
 \item[] \hspace*{20pt} $E X \ \leq \ b \ \sum_{j = 1}^N \ (j + 1) \ \leq \ (b/2)(N + 2)^2$. \vspace*{8pt}
 \item for all $a > 4$, we have \vspace*{4pt}
 \item[] \hspace*{20pt} $E X \ \leq \ b \,(a/4 - 1)^{-1} \ \sum_{j = 1}^N \ (a/4)^{j + 1} \ \leq \ b \,(a/4 - 1)^{-2} \,(a/4)^{N + 2}$.
\end{enumerate}
 Applying Markov's inequality, we deduce that
\begin{enumerate}
 \item for all $a < 4$, we have \vspace*{4pt}
 \item[] \hspace*{20pt} $P \,(X > M^{1/3}) \ \leq \ M^{-1/3} \ E X \ \leq \ M^{-1/3} \,b \,N (1 - a/4)^{-1}$. \vspace*{4pt}
 \item for $a = 4$, we have \vspace*{4pt}
 \item[] \hspace*{20pt} $P \,(X > M^{1/3}) \ \leq \ M^{-1/3} \ E X \ \leq \ M^{-1/3} \,(b/2)(N + 2)^2$. \vspace*{4pt}
 \item for all $a > 4$, we have \vspace*{4pt}
 \item[] \hspace*{20pt} $P \,(X > M^{1/3}) \ \leq \ M^{-1/3} \ E X \ \leq \ M^{-1/3} \,b \,(a/4 - 1)^{-2} \,(a/4)^{N + 2}$.
\end{enumerate}
 The lemma follows by combining these estimates together with \eqref{eq:range-collision-1}--\eqref{eq:range-collision-2}.
\end{proof} \\ \\
 Theorem \ref{th:range} directly follows from \eqref{eq:range-1} and Lemma \ref{lem:range-collision}.


\section{Proof of Theorem \ref{th:extinction}}
\label{sec:extinction}

\indent To prove Theorem \ref{th:extinction}, we return to the microscopic representation $\eta_{\cdot}$ of the process that looks at the
 metapopulation at the individual level rather than at the patch level.
 The proof is based on duality techniques so the first step is to define a graphical representation to construct both the microscopic process
 and its dual process. Recall that
 $$ \begin{array}{rcl}
     A (\x) & := & \hbox{set of potential parents' pairs within the patch containing $\x$} \vspace*{3pt} \\
            & := & \{(\y, \z) \in (\Z \times K_N)^2 : \y \neq \z \ \hbox{and} \ \pi (\x) = \pi (\y) = \pi (\z) \} \vspace*{9pt} \\
     B (\x) & := & \hbox{set of potential parents' pairs in the neighborhood of the patch containing $\x$} \vspace*{3pt} \\
            & := & \{(\y, \z) \in (\Z \times K_N)^2 : \y \neq \z \ \hbox{and} \ \pi (\x) \sim \pi (\y) = \pi (\z) \}. \end{array} $$
 The graphical representation consists of three collections of Poisson processes:
\begin{enumerate}
 \item For each $\x \in \Z \times K_N$ and each $(\y, \z) \in A (\x)$, we let \vspace*{4pt}
 \item[] \hspace*{20pt} $A_t (\x, \y, \z) \ := \ $ Poisson process with intensity $a \,(N (N - 1))^{-1}$ \vspace*{2pt}
 \item[] \hspace*{20pt} $a_n (\x, \y, \z) \ := \ $ the $n$th arrival time of $A_t (\x, \y, \z)$. \vspace*{8pt}
 \item For each $\x \in \Z \times K_N$ and each $(\y, \z) \in B (\x)$, we let \vspace*{4pt}
 \item[] \hspace*{20pt} $B_t (\x, \y, \z) \ := \ $ Poisson process with intensity $(b / 2M) \,(N (N - 1))^{-1}$ \vspace*{2pt}
 \item[] \hspace*{20pt} $b_n (\x, \y, \z) \ := \ $ the $n$th arrival time of $B_t (\x, \y, \z)$. \vspace*{8pt}
 \item For each $\x \in \Z \times K_N$, we let \vspace*{4pt}
 \item[] \hspace*{20pt} $D_t (\x) \ := \ $ Poisson process with intensity one \vspace*{2pt}
 \item[] \hspace*{20pt} $d_n (\x) \ := \ $ the $n$th arrival time of $D_t (\x)$.
\end{enumerate}
 The microscopic process is constructed from these Poisson processes as follows:
\begin{itemize}
 \item {\bf Births}: at time $t = a_n (\x, \y, \z)$ or $t = b_n (\x, \y, \z)$, we set
  $$ \eta_t (\x) := 1 \ \ \hbox{when} \ \ \eta_{t-} (\y) \,\eta_{t-} (\z) = 1 \quad \hbox{and} \quad
     \eta_t (\x) := \eta_{t-} (\x) \ \ \hbox{when} \ \ \eta_{t-} (\y) \,\eta_{t-} (\z) = 0. $$
 \item {\bf Deaths}: at time $t = d_n (\x)$, we set $\eta_t (\x) := 0$.
\end{itemize}
 From the graphical representation, we also construct
 $$ \hat \eta_s (\w, t) \ := \ \hbox{the dual process starting at $(\w, t)$} $$
 which allows to keep track of the state at $(\w, t)$ based on the configuration at time $t - s$.
 The dual process of the contact process with sexual reproduction consists of a collection of finite subsets of the set of
 sites $\Z \times K_N$ that evolves based on the graphical representation as follows:
\begin{enumerate}
\item The process starts from the singleton $\hat \eta_0 (\w, t) = \{\!\{\w \}\!\}$. \vspace*{4pt}
\item {\bf Births}: if site $\x \in \hat \eta_{s-} (\w, t)$ where either
 $$ s = t - a_n (\x, \y, \z) \quad \hbox{or} \quad s = t - b_n (\x, \y, \z) \qquad \hbox{for some} \ n > 0 $$
 then, for each set $B \in \hat \eta_{s-} (\w, t)$ that contains $\x$, we add the set which is obtained from~$B$ by
 removing $\x$ and adding its parents' sites $\y$ and $\z$, i.e.,
 $$ \hat \eta_s (\w, t) \ := \ \hat \eta_{s-} (\w, t) \ \cup \ \{(B - \{\x \}) \,\cup \,\{\y, \z \} : \x \in B \in \hat \eta_{s-} (\w, t) \}. $$
\item {\bf Deaths}: if site $\x \in \hat \eta_{s-} (\w, t)$ where $s = t - d_n (\x)$ for some $n > 0$ then we remove from the dual process
 all the sets that contain $\x$, i.e.,
 $$ \hat \eta_s (\w, t) \ := \ \hat \eta_{s-} (\w, t) - \{B \in \hat \eta_{s-} (\w, t) : \x \in B \}. $$
\end{enumerate}
 We say that a {\bf collision} occurs at dual time $s$ whenever
 $$ \hbox{the birth event 2 above occurs and} \ \hat \eta_{s-} (\w, t) \,\cap \,\{\y, \z \} \neq \varnothing. $$
 In words, there is a collision whenever there is a birth $\{\y, \z \} \to \x$ while $\x$ and at least one of its two
 parents $\y$ and $\z$ belong to the dual process.
 The dual process allows to deduce the state of site $\w$ at time $t$ from the configuration at earlier times.
 More precisely, identifying the microscopic process with the set of occupied sites, the construction of the dual process implies
 the following duality relationship between the two processes:
\begin{equation}
\label{eq:duality-1}
  \w \in \eta_t \quad \hbox{if and only if} \quad B \,\subset \,\eta_{t - s} \ \ \hbox{for some} \ \ B \in \hat \eta_s (\w, t).
\end{equation}
 In view of \eqref{eq:duality-1}, regardless of the initial configuration $\eta_0$,
\begin{equation}
\label{eq:duality-2}
  \eta_t (\w) = 0 \quad \hbox{whenever} \quad \hat \eta_t (\w, t) = \varnothing
\end{equation}
 which is the key to proving the theorem.
 The dual process of the contact process with sexual reproduction is a very complicated object and finding even rough conditions on the
 parameters of the system for almost sure extinction of the dual process is a challenging problem.
 In the absence of collision, however, the problem is simplified due to some events becoming independent.
 In particular, the next step is to show that the probability of a collision before a fixed dual time can be made arbitrarily small by
 choosing $N$ sufficiently large, which is done in the following lemma.
\begin{lemma} --
\label{lem:extinction-collision}
 For all $\ep, t > 0$ there exists $N$ sufficiently large such that the probability that a collision occurs before dual time $t$ is
 smaller than $(1/2) \ep$.
\end{lemma}
\begin{proof}
 The basic idea is that the number of points $N_s$ added to the dual process in $s$ time units does not depend on the parameter $N$.
 More precisely, the random variable $N_s$ is stochastically smaller than the number of particles at time $s$ in a branching
 process $X_s$ with no death and in which each particle gives birth to a pair of particles at rate $a + b$ therefore
 $$ E N_s \ \leq \ E X_s \ \leq \ E X_t \ \leq \ \exp \,(2 (a + b) \,t) \quad \hbox{for all} \ s \in (0, t). $$
 Using Markov's inequality further implies that
\begin{equation}
\label{eq:extinction-collision-1}
  P \,(N_s > K \ \hbox{for some} \ s < t) \ \leq \ P \,(X_t > K) \ \leq \ K^{-1} \,E X_t \ \leq \ K^{-1} \,\exp \,(2 (a + b) \,t).
\end{equation}
 Since on the event that $N_s \leq K$ for all $s < t$ there are at most $K$ births and since at each birth the parents are from
 the set of all pairs among at least $N$ vertices, we also have
\begin{equation}
\label{eq:extinction-collision-2}
 \begin{array}{l}
   P \,(\hbox{collision before dual time $t$} \ | \ N_s \leq K \ \hbox{for all} \ s < t) \vspace*{4pt} \\ \hspace*{40pt} \leq \
   K \,(1 - (N - K)(N - K - 1) \,N^{-1} (N - 1)^{-1}) \ \leq \ 2K (K + 1) N^{-1}. \end{array}
\end{equation}
 Combining \eqref{eq:extinction-collision-1}--\eqref{eq:extinction-collision-2}, we get
 $$ \begin{array}{l}
     P \,(\hbox{collision before dual time $t$}) \vspace*{4pt} \\ \hspace*{40pt} = \
     P \,(\hbox{collision before dual time $t$} \ \hbox{and} \ N_s > K \ \hbox{for some} \ s < t) \vspace*{4pt} \\ \hspace*{80pt} + \
     P \,(\hbox{collision before dual time $t$} \ \hbox{and} \ N_t \leq K \ \hbox{for all} \ s < t) \vspace*{4pt} \\ \hspace*{40pt} \leq \
     P \,(N_s > K \ \hbox{for some} \ s < t) \vspace*{4pt} \\ \hspace*{80pt} + \
     P \,(\hbox{collision before dual time $t$} \ | \ N_s \leq K \ \hbox{for all} \ s < t) \vspace*{4pt} \\ \hspace*{40pt} \leq \
     K^{-1} \,\exp \,(2 (a + b) \,t) + 2K (K + 1) N^{-1} \end{array} $$
 which, taking $K = N^{0.2}$, tends to zero as $N \to \infty$.
\end{proof} \\ \\
 Returning to our key ingredient \eqref{eq:duality-2}, the next step is to prove almost sure extinction under the conditions of the
 theorem of a version of the dual process that excludes collisions.
 This version is constructed from the following random variables: for all $u \in (0, 1)$ and $n > 0$, we let
\begin{enumerate}
 \item $\mathbf b_n (u)$ = the $n$th arrival time of a Poisson process with intensity $a + b$. \vspace*{4pt}
 \item $\mathbf d_n (u)$ = the $n$th arrival time of a Poisson process with intensity one. \vspace*{4pt}
 \item $\mathbf U_n^+ (u)$ and $\mathbf U_n^- (u)$ are independent $\uniform (0, 1)$.
\end{enumerate}
 For each $(w, t) \in (0, 1) \times \R_+$, we construct
 $$ \hat \zeta_s (w, t) \ := \ \hbox{version of the dual process with no collision starting at $(w, t)$}. $$
 This process consists of a collection of finite subsets of the unit interval $(0, 1)$.
\begin{enumerate}
\item The process starts from the singleton $\hat \zeta_0 (w, t) = \{\!\{w \}\!\}$. \vspace*{4pt}
\item {\bf Births}: if point $u \in \hat \zeta_{s-} (w, t)$ where
 $$ s = t - \mathbf b_n (u) \quad \hbox{for some} \ n > 0 $$
 then, for each set $B \in \hat \zeta_{s-} (w, t)$ that contains $u$, we add the set which is obtained from~$B$ by
 removing $u$ and adding the two points $\mathbf U_n^+ (u)$ and $\mathbf U_n^- (u)$, i.e.,
 $$ \hat \zeta_s (w, t) \ := \ \hat \zeta_{s-} (w, t) \ \cup \ \{(B - \{u \}) \,\cup \,\{\mathbf U_n^+ (u), \mathbf U_n^- (u) \} : u \in B \in \hat \zeta_{s-} (w, t) \}. $$
\item {\bf Deaths}: if point $u \in \hat \zeta_{s-} (w, t)$ where
 $$ s = t - \mathbf d_n (u) \quad \hbox{for some} \ n > 0 $$
 then we remove all the sets that contain $u$, i.e.,
 $$ \hat \zeta_s (w, t) \ := \ \hat \zeta_{s-} (w, t) - \{B \in \hat \zeta_{s-} (w, t) : u \in B \}. $$
\end{enumerate}
 The process evolves according to the exact same rules as the dual process except that at each birth event the parents
 are chosen uniformly at random in the unit interval rather than from a finite collection of neighbors.
 In particular, defining collisions as for the dual process, collisions now occur with probability zero, and it directly
 follows from the construction that
\begin{equation}
\label{eq:duality-coupling}
  P \,(\hat \eta_t (\w, t) = \varnothing \ | \ \hbox{no collision before dual time $t$}) \ = \ P \,(\hat \zeta_t (w, t) = \varnothing).
\end{equation}
 Both this process and the dual process are naturally defined only for $0 \leq s \leq t$ but it is convenient to assume
 that all the Poisson processes introduced above are also defined for negative times so that both processes are well defined
 for all $s \geq 0$.
 Then, we say that
 $$ \begin{array}{rrcl}
    (w, t) \in (0, 1) \times \R_+ & \hbox{{\bf lives forever}} & \hbox{when} & \hat \zeta_s (w, t) \neq \varnothing \ \hbox{for all} \ s \geq 0 \vspace*{2pt} \\
                                  & \hbox{{\bf dies out}}      & \hbox{when} & \hat \zeta_s (w, t)   =  \varnothing \ \hbox{for some} \ s \geq 0. \end{array} $$
 and similarly for the dual process:
 $$ \begin{array}{rrcl}
    (\w, t) \in (\Z \times K_N) \times \R_+ & \hbox{{\bf lives forever}} & \hbox{when} & \hat \eta_s (\w, t) \neq \varnothing \ \hbox{for all} \ s \geq 0 \vspace*{2pt} \\
                                            & \hbox{{\bf dies out}}      & \hbox{when} & \hat \eta_s (\w, t)   =  \varnothing \ \hbox{for some} \ s \geq 0. \end{array} $$
 Due to the lack of collision and independence of some key events, the exact condition for~$(w, t)$ to die out with probability
 one can be found, which is done in the next lemma.
\begin{lemma} --
\label{lem:extinction-dual}
 Assume that $a + b < 4$.
 Then, $P \,((w, t) \ \hbox{dies out}) = 1$.
\end{lemma}
\begin{proof}
 The proof relies on three ingredients.
 First, due to symmetry,
\begin{equation}
\label{eq:extinction-dual-1}
  \rho (w, t) \ := \ P \,((w, t) \ \hbox{lives forever}) \ = \ \rho \ = \ \hbox{constant}
\end{equation}
 is constant across all the possible choices of the starting point.
 Second, due to the absence of collision, the events that different space-time points live forever are independent:
\begin{equation}
\label{eq:extinction-dual-2}
 \begin{array}{l}
   P \,((w_1, t) \ \hbox{and} \ (w_2, t) \ \hbox{live forever}) \vspace*{4pt} \\ \hspace*{50pt} = \
   P \,((w_1, t) \ \hbox{lives forever}) \ P \,((w_2, t) \ \hbox{lives forever}). \end{array}
\end{equation}
 Third, by construction of the process, for all $s > 0$,
\begin{equation}
\label{eq:extinction-dual-3}
 \begin{array}{l}
   (w, t) \ \hbox{lives forever if and only if} \vspace*{2pt} \\ \hspace*{30pt}
            \hbox{there exists} \ B \in \hat \zeta_s (w, t) \ \hbox{such that} \ (u, t - s) \ \hbox{lives forever for all} \ u \in B. \end{array}
\end{equation}
 The lemma follows from \eqref{eq:extinction-dual-1}--\eqref{eq:extinction-dual-3} combined with a first-step analysis.
 The idea is to condition on whether the first event occurring at point $w$ before time $t$ is a birth or a death.
 The first time before time $t$ a birth or death occur at $w$ are given respectively by
 $$ \begin{array}{rcl}
     s_+ & := & t - \max \,(\{\mathbf b_n (w) : n \in \Z \} \,\cap \,(- \infty, t)) \vspace*{2pt} \\
     s_- & := & t - \max \,(\{\mathbf d_n (w) : n \in \Z \} \,\cap \,(- \infty, t)). \end{array} $$
 Letting $s_0 = \min \,(s_+, s_-)$, we have
\begin{equation}
\label{eq:extinction-dual-4}
  \hat \zeta_{s_0} (w, t) = \{\{w \}, \{w_1, w_2 \} \} \ \hbox{when} \ s_0 = s_+ \quad \hbox{and} \quad
  \hat \zeta_{s_0} (w, t) = \varnothing \ \hbox{when} \ s_0 = s_-
\end{equation}
 where points $w_1$ and $w_2$ are the parents of point $w$ at dual time $s_+$.
 In particular, combining the two properties of the process \eqref{eq:extinction-dual-3}--\eqref{eq:extinction-dual-4}, we get:
 $(w, t)$ lives forever if and only if
\begin{enumerate}
 \item $s_0 = s_+$ and \vspace*{4pt}
 \item ($(w, t - s_0)$ lives forever) or ($(w_1, t - s_0)$ and $(w_2, t - s_0)$ both live forever).
\end{enumerate}
 The first event has probability
\begin{equation}
\label{eq:extinction-dual-5}
  P \,(s_0 = s_+) \ = \ (a + b)(1 + a + b)^{-1}.
\end{equation}
 Using \eqref{eq:extinction-dual-1}--\eqref{eq:extinction-dual-2}, we also have
\begin{equation}
\label{eq:extinction-dual-6}
 \begin{array}{l}
   P \,((w, t - s_0) \ \hbox{lives forever or} \ (w_1, t - s_0) \ \hbox{and} \ (w_2, t - s_0) \ \hbox{both live forever}) \vspace*{4pt} \\ \hspace*{25pt} = \
   1 - P \,((w, t - s_0) \ \hbox{dies out}) \,P \,((w_1, t - s_0) \ \hbox{or} \ (w_2, t - s_0) \ \hbox{dies out}) \vspace*{4pt} \\ \hspace*{25pt} = \
   1 - P \,((w, t - s_0) \ \hbox{dies out}) \vspace*{4pt} \\ \hspace*{80pt}
      (1 - P \,((w_1, t - s_0) \ \hbox{and} \ (w_2, t - s_0) \ \hbox{both live forever})) \vspace*{4pt} \\ \hspace*{25pt} = \
   1 - (1 - \rho)(1 - \rho^2) \ = \ \rho^2 \,(1 - \rho) + \rho.
 \end{array}
\end{equation}
 Finally, combining \eqref{eq:extinction-dual-5}--\eqref{eq:extinction-dual-6}, we deduce that
 $$ \begin{array}{rcl}
     (1 + a + b) \,\rho & := & (1 + a + b) \ P \,((w, t) \ \hbox{lives forever}) \vspace*{4pt} \\
                        &  = & (1 + a + b) \ P \,((w, t) \ \hbox{lives forever} \ | \,s_0 = s_+) \,P \,(s_0 = s_+) \vspace*{4pt} \\
                        &  = & (a + b) \,(\rho^2 \,(1 - \rho) + \rho) \end{array} $$
 from which it follows that
 $$ (a + b) \,\rho^2 \,(1 - \rho) - \rho \ = \ (a + b) \,(\rho^2 \,(1 - \rho) + \rho) - (1 + a + b) \,\rho \ = \ 0. $$
 In particular, $\rho = 0$ whenever $a + b < 4$.
\end{proof} \\ \\
 Theorem \ref{th:extinction} directly follows from Lemmas~\ref{lem:extinction-collision}--\ref{lem:extinction-dual} and the duality relationship \eqref{eq:duality-2}.
\begin{lemma} --
\label{lem:extinction}
 Let $\w \in \Z \times K_N$ and $a + b < 4$.
 For all $\ep > 0$, there exists $t$ such that
 $$ P \,(\eta_t (\w) = 1) \ \leq \ \ep \quad \hbox{for all $N$ sufficiently large}. $$
\end{lemma}
\begin{proof}
 Since $a + b < 4$, according to Lemma~\ref{lem:extinction-dual}, the probability that any given point $(w, t)$ lives forever is equal
 to zero hence there exists $t$ fixed from now on such that
\begin{equation}
\label{eq:extinction-1}
  P \,(\hat \zeta_t (w, t) \neq \varnothing) \ \leq \ (1/2) \ep.
\end{equation}
 In addition, time $t$ being fixed, Lemma~\ref{lem:extinction-collision} implies that
\begin{equation}
\label{eq:extinction-2}
  P \,(\hbox{collision before dual time $t$}) \ \leq \ (1/2) \ep \quad \hbox{for all $N$ large}.
\end{equation}
 Combining \eqref{eq:extinction-1}--\eqref{eq:extinction-2} with \eqref{eq:duality-coupling}, we obtain
 $$ \begin{array}{rcl}
      P \,(\hat \eta_t (\w, t) \neq \varnothing) & = &
      P \,(\hat \eta_t (\w, t) \neq \varnothing \ \hbox{and no collision before dual time $t$}) \vspace*{4pt} \\ && \hspace*{25pt} + \
      P \,(\hat \eta_t (\w, t) \neq \varnothing \ \hbox{and collision before dual time $t$}) \vspace*{4pt} \\ & \leq &
      P \,(\hat \eta_t (\w, t) \neq \varnothing \ | \ \hbox{no collision before dual time $t$}) \vspace*{4pt} \\ && \hspace*{25pt} + \
      P \,(\hbox{collision before dual time $t$}) \vspace*{4pt} \\ & = &
      P \,(\hat \zeta_t (w, t) \neq \varnothing) + P \,(\hbox{collision before dual time $t$}) \ \leq \ \ep. \end{array} $$
 The theorem follows from \eqref{eq:duality-2} and the previous inequality.
\end{proof}


\end{document}